\documentclass[preprint,12pt]{elsarticle}

\usepackage{amsmath,amsthm,amssymb, latexsym,mathrsfs}

\usepackage{amssymb}

 \newtheorem{thm}{Theorem}[section]
 
 \newtheorem{lem}[thm]{Lemma}
 \newtheorem{prop}[thm]{Proposition}
 \theoremstyle{definition}

 \newtheorem{rmk}[thm]{Remark}
 \newtheorem{ass}[thm]{Assumption}
 \numberwithin{equation}{section}
\newcommand{\C}{{\mathbb C}}

\newcommand{\Z}{{\mathbb Z}}
\newcommand{\diag}{\text{\rm diag}}
\newcommand{\sgn}{\text{\rm sgn}}

\begin{document}

\begin{frontmatter}


 \tnotetext[label1]{}
 \author{Bertin Zinsou\corref{cor1}\fnref{label2}}
 \ead{bertin.zinsou@wits.ac.za}
 \ead[url]{www.wits.ac.za}
 \cortext[cor1]{}

\title{Stability of a flexible missile described by asymptotics of the eigenvalues of fourth order  boundary value   problems}



\address{School of Mathematics, University of the Witwatersrand, Private Bag 3 Wits 2050, Johannesburg, South Africa }

\begin{abstract}

Fourth order problems, with the differential equation $y^{(4)}-(gy')'=\lambda^2y$, where $g\in C^1[0,a]$ and $a>0$, occur  in engineering on stability of elastic rods. They occur as well in aeronautics to describe the stability of a flexible missile. 
Fourth order Birkhoff regular problems  with the differential equation $y^{(4)}-(gy')'=\lambda^2y$  and  eigenvalue dependent boundary conditions  are considered. These problems have quadratic  operator representations  with non self-adjoint operators. The first four terms of the asymptotics of the eigenvalues of the problems  as well as  those of the eigenvalues of the  problem describing the stability of a flexible missile are   evaluated explicitly.\end{abstract}


\end{frontmatter}

\section{Introduction}
 Higher order ordinary differential operators occur    in applications with or without the eigenvalue parameter in the boundary conditions. Such problems are realized as operator polynomials, also called operator pencils. Some recent developments of higher order differential operators whose boundary conditions may depend on the eigenvalue parameter have been investigated in \cite{ MoletZin, MolPiv,MolZin,MolZin1,MolZin4,MolZin5,wan, Zin,  MolZin6}.

Problems like the generalized  Regge problem, the stability of elastic rod     problems and the vibrating curve  problems have boundary conditions with partial first derivatives with respect to the time variable $t$ or whose mathematical model leads to an eigenvalue problem with the eigenvalue parameter $\lambda$ occurring linearly in the boundary conditions. Such problems have an operator representation of the form
\begin{equation}\label{L(lambda)}L(\lambda)=\lambda^2M-i\lambda K-A\end{equation}
in the Hilbert space $H=L_2(I)\oplus\mathbb{C}^k$, where $I$ is an interval,  $k$ the number of eigenvalue dependent boundary conditions, $M$, $K$ and $A$ are coefficient operators.

Separation of variables leads  the stability of elastic rod  problems investigated in \cite{MoletZin, MolPiv, MolZin, MolZin1, Zin, MolZin6} to fourth order eigenvalue problems with eigenvalue parameter dependent boundary conditions, where the differential equation 
\begin{equation}\label{difEq}y^{(4)}-(gy')'=\lambda^2y\end{equation} depends on the eigenvalue parameter. The  fourth order problem  with the differential equation \eqref{difEq} and the boundary conditions $y''(0)=y^{(3)}(0)=0$ and  $y''(a)=y^{(3)}(a)=0$   describes the stability of a flexible missile, see \cite{Beal,Guran,  KirilSeyra}. This problem can be represented by the operator polynomial 
\begin{equation}\label{L(lambdaFlex)}L(\lambda)=\lambda^2M-A\end{equation}
in the Hilbert space $L_2(I)$.

 In \cite{MolZin} we have investigated a class of boundary conditions for which necessary and sufficient conditions have been obtained such that the associated operator pencil consists of self-adjoint operators, while in \cite{MolZin1} we have continued the work of \cite{MolZin} in the direction of \cite{MolPiv} to  derive eigenvalue asymptotics associated with boundary conditions which do not depend on the eigenvalue parameter at the left endpoint and depend on the eigenvalue parameter at the right endpoint. Note that the problems investigated in \cite{MolPiv, MolZin1} are Birkhoff regular. In \cite{Zin} we have investigated a class of boundary conditions for which necessary and sufficient conditions have been obtained such that the associated operator pencil is Birkhoff regular. 

In this paper we extend the work of \cite{MolZin1} to   classes of Birkhoff regular problems  where the coefficient operators $K$ and $A$  of the associated quadratic operator pencil are not necessary self-adjoint.

We give a characterization of   fourth order Birkhoff regular  problems  in Section \ref{BikSec}. In  Section \ref{penc} we present  the quadratic  operator pencil under consideration  as well as the boundary conditions  that will be investigated. In Section \ref{s:g=0}  we classify  the eigenvalue dependent boundary problems under consideration in two different classes according to the right endpoint boundary conditions and we derive the eigenvalue asymptotics for the case $g=0$. As these problems     are  Birkhoff regular, then the eigenvalues for general $g$ are small perturbations of those for $g=0$. Hence in Section \ref{gen:g}  we use the eigenvalue asymptotics for $g=0$ to provide the first four terms of the eigenvalue asymptotics of the two relevant   classes and  we compare the results obtained  to those obtained  in \cite{MolZin1}. Finally in Section 
\ref{flex} we give the asymptotics of the eigenvalues of  the problem describing the stability of a flexible missile. 


\section{Fourth order Birkhoff regular problems}\label{BikSec}

On the interval $[0,a]$, we    consider the eigenvalue problem
 \begin{gather}\label{eq:2inteq0} y^{(4)} - (gy')' = \lambda^2y,\\
B_j(\lambda)y = 0,\ j=1,2,3,4,\label{eq:2cneq1}\end{gather}
where    $g \in C^1[0, a]$, $a>0$, is a real valued function and \eqref{eq:2cneq1}  are  separated  boundary conditions   independent of $\lambda$ or depending on $\lambda$ linearly.
We assume  that    
 
\begin{equation}\label{quasi}B_j(\lambda)y=y^{[p_j]}(a_j)+ i\beta_j\lambda y^{[q_j]}(a_j),\end{equation}
where $a_j=0$ for $j=1,2$ and $a_j=a$ for $j=3,4$, with $0\le q_j<p_j\le 3$, for  $\beta_j\in\mathbb{C}\setminus\{0\}$  while  $\beta_j=0$ corresponds to $q_j=-\infty$, $j=1,2,3,4$. 

We recall that the quasi-derivatives associated to \eqref{eq:2inteq0}  are given by 
\begin{align}\label{quasiEq}y^{[0]}=y,~y^{[1]}=y', ~ y^{[2]}=y'', ~ y^{[3]}=y^{(3)}-gy',~y^{[4]}=y^{(4)}-(gy')',\end{align}  see \cite[Definition 10.2.1, page 256]{MolPiv1}.

Recall that in applications, using separation of variables, the parameter $\lambda$ emanates from derivatives with respect to the time variable in the original partial differential equation, and it is reasonable that the highest space derivative occurs in the term without time derivative. Thus the most relevant boundary conditions would have $q_j<p_j$ for $j=1,2,3,4$.

 We define
 \begin{gather}\label{theta}\Theta_1 = \left\{s \in \{1, 2, 3, 4\} : B_s(\lambda)  \textrm{ depends on } \lambda \right\},
 \ \Theta_0 = \{1, 2, 3, 4\}\backslash \Theta_1,  \\\label{theta0}
\Theta_1^0 = \Theta_1\cap \{1, 2\}, \quad \Theta_1^a = \Theta_1\cap \{3, 4\}, \end{gather}
and
\begin{gather}\label{lambda}\Lambda= \left\{s \in \{1, 2, 3, 4\} : p_s>-\infty \right\}, ~ \Lambda^0=\Lambda\cap\{1,2\},~ \Lambda^a=\Lambda\cap\{3,4\}.\end{gather}

 \begin{ass}\label{ass:ass1}
We assume that   the numbers $p_s$ for $s\in\Lambda^0$,  $q_j$  for $j \in \Theta_1^0$   are distinct and
that   the numbers $p_s$ for $s\in\Lambda^a$,   $q_j$ for $j \in \Theta_1^a$ are distinct.
\end{ass}
Assumption \ref{ass:ass1} means that for any pair $(r, a_j)$ the term $y^{[r]}(a_j)$ occurs at most once in the boundary conditions \eqref{eq:2cneq1} and  that the numbers $q_j$, $p_j$, $j=1,2,3,4$ are mutually disjoint.

Let $p_j,q_j\in\{0,1,2,3\}$, where $p_j,q_j$ are as defined in Assumption \ref{ass:ass1}, $j=1,2,3,4$. Let $u$ such that $u=0$ if $j=1,2$ and $u=1$ if $j=3,4$. Let $C(r,u)$, $r=1,2,3,4,5$, $u=0,1$, be the following conditions:\\
$C(1,u)$: $p_{1+2u}>q_{1+2u}+2$, $p_{2+2u}>q_{2+2u}+2$;\\ 
$C(2,u)$: $p_{1+2u}>q_{1+2u}+2$, $q_{2+2u}+2> p_{2+2u}$;\\
 $C(3,u)$: $p_{1+2u}>q_{1+2u}+2$,   $ p_{2+2u}=q_{2+2u}+2$ and  
$\beta_{2+2u}\ne (-1)^{l},$  where $l=1,2$;\\
$C(4,u)$: $q_{1+2u}+2> p_{1+2u}$, $q_{2+2u}+2> p_{2+2u}$; \\
 $C(5,u)$: $q_{1+2u}+2> p_{1+2u}$, $p_{2+2u}=q_{2+2u}+2$,
\begin{gather*}\beta_{2+2u}\ne\begin{cases}(-1)^{l+1}  \quad \textrm{if}\quad q_{1+2u}-q_{2+2u}=1,\\   (-1)^{l}  \qquad \textrm{if}\quad q_{1+2u}-q_{2+2u}=3,\end{cases}\end{gather*}   where $l=1,2$.

For the boundary conditions \eqref{eq:2cneq1} and the assumptions made so far, \cite[Theorem 3.4]{Zin} leads to the following.    
\begin{prop}\label{main}  The problem  \eqref{eq:2inteq0}, \eqref{eq:2cneq1} is  Birkhoff regular if and only if there are $r_0, r_1\in\left\{1,2,3,4,5 \right\}$ such that the conditions $C(r_0,0)$ and $C(r_1,1)$ hold.\end{prop}



\section{The quadratic operator pencil $L$}\label{penc}

We denote the collection of boundary conditions \eqref{eq:2cneq1} by $U$ and define the following operators related to $U$
\begin{gather}\label{OperU}U_ry=\begin{pmatrix} y^{[p_j] }(a_j)\end{pmatrix}_{j\in\Theta_r}, r=0,1, ~\textrm{and}~ V_1y=\begin{pmatrix} \beta_jy^{[q_j]}(a_j)\end{pmatrix}_{j\in\Theta_1},  \\ y\in W_2^4(0,a),\nonumber\end{gather}
where $W_2^4(0,a)$ is the Sobolev space of order $4$ on the interval $(0,a)$. 

We put $k=|\Theta_1|$ and we  consider the linear operators $A(U)$,  $K$  and $M$ in the space $L_2(0,a)\oplus \mathbb{C}^{k}$ with domains
\begin{align*}&\mathscr{D}(A(U))=\left\{\widetilde y=\begin{pmatrix}y\\V_1y\end{pmatrix}: y\in W_2^4(0,a),  U_0y=0\right\},\\&\mathscr{D}(K)=\mathscr{D}(M)=L_2(0,a)\oplus\mathbb{C}^k,\end{align*}
given by 
\begin{gather*}(A(U))\widetilde{y}=\left(\begin{matrix}y^{(4)}-(gy')'\\U_1y\end{matrix}\right) ~ \textrm{for } \widetilde y\in\mathscr{D}(A(U)), \\  M=\begin{pmatrix}I&0\\0&0\end{pmatrix} \quad ~\textrm{and}~\quad  K=\begin{pmatrix}0&0\\0&K_0\end{pmatrix} ~\textrm{with}~K_0=\diag(\beta_j : j\in\Theta_1).\end{gather*}
It is clear that $M$ and $K$ are bounded  operators and   $M$  is nonnegative and self-adjoint.   We associate a quadratic operator pencil
\begin{align}\label{OperaPencil}L(\lambda)=\lambda^2M- i  \lambda K-A(U), \quad \lambda\in\mathbb{C}\end{align}
in the space $L_2(0,a)\oplus\mathbb{C}^k$ with the problems \eqref{eq:2inteq0}, \eqref{eq:2cneq1}. We observe that \eqref{OperaPencil} is an operator representation of the eigenvalue problem \eqref{eq:2inteq0}, \eqref{eq:2cneq1} in the sense that a function $y$ satisfies \eqref{eq:2inteq0}, \eqref{eq:2cneq1} if and only if it satisfies 
$L(\lambda)\widetilde y=0$.

Note that if all the boundary conditions in \eqref{eq:2cneq1} are independent of $\lambda$, then $V_1y=0$ and $U_1y=0$, where $y\in W_2^4(0,a)$. Hence \eqref{OperaPencil} will be reduced to 
\begin{align}\label{OperaPencil0}L(\lambda)=\lambda^2M   -A(U), \quad \lambda\in\mathbb{C}\end{align} in the space $L_2(0,a)$.

We are going to investigate the asymptotics of the eigenvalues of  the classes of the boundary value problems where the boundary conditions at the left endpoint are independent of the parameter  $\lambda$, while the boundary conditions at the right endpoint depend or may not depend on the parameter. For the case $\beta_3\beta_4\ne0$, we are going   to compare the results of our investigation  to those obtained in the case  of self-adjoint problems studied in \cite{MolZin1}.
Hence the four  boundary conditions    \eqref{eq:2cneq1} are
\begin{gather}\label{case1}\begin{cases}y^{[p_1]}(0)=0,\quad y^{[p_2]}(0)=0,\\ y^{[p_3]}(a)+i\beta_3\lambda y^{[q_3]}(a)=0, \quad y^{[p_4]}(a)+i\beta_4\lambda y^{[q_4]}(a)=0,\end{cases}\end{gather} 
where  $0\le p_1<p_2\le 3$,   $0\le q_3<p_3\le 3$, $0\le q_4<p_4\le 3$  and  $0< p_3<p_4\le 3$. 
Therefore taking Assumption \ref{ass:ass1} into account,  we will distinguish the following different cases of boundary conditions at the endpoint $0$:\\
\begin{gather}\begin{cases} \rm Case ~ 1:(p_1, p_2)=(0,   1),  ~Case   ~2: ( p_1,p_2)=(0, 2), \\\rm Case   ~3:  (p_1,p_2)=(0,  3), ~
 Case  ~4: ( p_1,p_2)=(1,  2), \\\rm Case   ~5: (p_1,p_2)=(1,  3),~ Case   ~6: (p_1,p_2)=(2,   3). \end{cases}\end{gather}
However the boundary conditions at the right endpoint $a$ will be classified as
\begin{gather}\label{allclasses}\begin{cases}
 \rm Case^{(a)}~ 1: (p_3,q_3)=(1,0) ~and~  (p_4,q_4)=(3,2),\\
\rm Case^{(a)} ~2:  (p_3,q_3)=(2,1)  ~and  ~(p_4,q_4)=(3,0).\end{cases}\end{gather}
 As we have 2 sets of  boundary conditions at the endpoint $a$ and 6 sets of  boundary conditions at the endpoint $0$, then we have 12 sets of boundary conditions in total. We are going to  classify these 12 sets of boundary conditions according to the endpoint $a$. Hence we will have 2  classes of boundary conditions that we are going to classify by the pair $(p_j,q_j)$, $j=3,4$, see \eqref{allclasses}.

Define the condition $C'(2,u)$: $p_{1+2u}<q_{1+2u}+2$, $q_{2+2u}+2< p_{2+2u}$, $u=0,1$. Note that the conditions $C(2,u)$ and $C'(2,u)$, $u=0,1$  are redundant, see \cite[page 5]{Zin}. Hence for $u=0,1$, any result that is valid for $C(2,u)$, the equivalent result  is valid for $C'(2,u)$,  as  well.

Note that the left endpoint boundary conditions satisfy the condition $C(1,0)$, while the right endpoint boundary conditions satisfy the conditions $C(4,1)$ for Case$^{(a)}$ 1 and   the condition $C'(2,1)$ for Case$^{(a)}$ 2. Whence the problems are Birkhoff regular  for the classes  Case$^{(a)}$ 1 and Case$^{(a)}$ 2,    see Proposition \ref{main}.

We are going to investigate as well the asymptotics of the eigenvalues of the problem describing the stability of a flexible missile, where the boundary conditions are $y''(0)=y^{(3)}(0)=0$ and  $y''(a)=y^{(3)}(a)=0$. 

Note that the left endpoint boundary conditions of this problem satisfy the condition $C(1,0)$, while  the right endpoint boundary conditions satisfy the condition  $C(1,1)$. Hence the problem is Birkhoff regular according to Proposition  \ref{main}.



 \section{Asymptotics of eigenvalues for $g=0$}\label{s:g=0}
 In this section we consider the boundary value problems \eqref{eq:2inteq0}, \eqref{case1}   with $g=0$.  
We count all the eigenvalues with their proper multiplicities and develop a formula for
the asymptotic distribution of the eigenvalues for $g=0$, which is used to obtain the corresponding
formula for general $g$. We take the canonical fundamental system $y_j$, $j=1,\dots,4$, of \eqref{eq:2inteq0}  with $y_j^{(m)}(0)=\delta_{j,m+1}$ for $m=0,\dots,3$, which is analytic on $\mathbb{C}$ with respect
to $\lambda$. Putting
  \[M(\lambda )=(B_i(\lambda )y_j(\cdot ,\lambda ))_{i,j=1}^4,\]
the eigenvalues of the boundary value problems \eqref{eq:2inteq0}, \eqref{case1}   for $g=0$,       are the eigenvalues of the analytic matrix function $M$, where the corresponding geometric and algebraic multiplicities coincide, see \cite[Theorem 3.1.2]{men-mol}. Setting $\lambda =\mu^2$ and
  \[y(x,\mu)= \frac1{2\mu^3}\sinh(\mu x)-\frac1{2\mu^3}\sin(\mu x),\]
it is easy to see that
\begin{equation}\label{jDerivative}
  y_j(x,\lambda )=y^{(4-j)}(x,\mu),\quad j=1,\dots,4.
\end{equation}

Since the first and the second rows of $M(\lambda)$ have exactly one entry 1 and all other entries 0, it follows that for each of the 2  different classes of boundary conditions $\det~M(\lambda)=\pm \phi(\mu)$, where 
\begin{align*}\phi(\mu)=\det \begin{pmatrix}B_3(\mu^2)y_{\sigma(1)}(\cdot,\mu)&B_3(\mu^2)y_{\sigma(2)}(\cdot,\mu)\\B_4(\mu^2)y_{\sigma(1)}(\cdot,\mu)&B_4(\mu^2)y_{\sigma(2)}(\cdot,\mu)\end{pmatrix}.\end{align*} 
 \begin{gather}\label{Dcases}(\sigma(1),\sigma(2))=\begin{cases}(3,4) \textrm{ in  Case 1}, ~  (2,4) \textrm{ in Case 2},   (2, 3)~ \textrm{ in  Case 3},  \\(1,4) \textrm{ in  Case 4,} ~ (1,3)~\textrm{ in  Case 5, }    (1,2) \textrm{ in Case 6}.\end{cases}\end{gather}
Therefore
\begin{align}\label{phi}\phi(\mu)&=B_3(\mu^2)y_{\sigma(1)}(\cdot,\mu)B_4(\mu^2)y_{\sigma(2)}(\cdot,\mu)-B_4(\mu^2)y_{\sigma(1)}(\cdot,\mu)B_3(\mu^2)y_{\sigma(2)}(\cdot,\mu)\nonumber\\
&=\left(y_{\sigma(1)}^{(p_3)}(a)+i\beta_3\mu^2 y_{\sigma(1)}^{(q_3)}(a)\right)\left(y_{\sigma(2)}^{(p_4)}(a)+i\beta_4\mu^2 y_{\sigma(2)}^{(q_4)}(a)\right)\nonumber\\&\quad-\left(y_{\sigma(1)}^{(p_4)}(a)+i\beta_4\mu^2 y_{\sigma(1)}^{(q_4)}(a)\right)\left(y_{\sigma(2)}^{(p_3)}(a)+i\beta_3\mu^2 y_{\sigma(2)}^{(q_3)}(a)\right)\nonumber\\&=y_{\sigma(1)}^{(p_3)}(a)y_{\sigma(2)}^{(p_4)}(a)-y_{\sigma(1)}^{(p_4)}(a)y_{\sigma(2)}^{(p_3)}(a)+i\mu^2\left[\beta_3\Big(y_{\sigma(1)}^{(q_3)}(a)y_{\sigma(2)}^{(p_4)}(a)\right.\nonumber\\&\quad\left.-y_{\sigma(1)}^{(p_4)}(a)y_{\sigma(2)}^{(q_3)}(a)\Big)+\beta_4\Big(y_{\sigma(1)}^{(p_3)}(a)y_{\sigma(2)}^{(q_4)}(a)-y_{\sigma(1)}^{(q_4)}(a)y_{\sigma(2)}^{(p_3)}(a)\Big)\right]\nonumber\\&\quad+\beta_3\beta_4\mu^4\Big[y_{\sigma(1)}^{(q_4)}(a)y_{\sigma(2)}^{(q_3)}(a)-y_{\sigma(1)}^{(q_3)}(a)y_{\sigma(2)}^{(q_4)}(a)\Big].\end{align}
Next we discuss the asymptotics of the zeros of the problems for each  class Case$^{(a)}$ $j$, $j=1,2$. 
 \subsection{Asymptotics of eigenvalues for $g=0$ of the  problems of Class Case$^{(a)}$ 1}\label{Class1:g=0}
It follows from \eqref{allclasses} and \eqref{phi} that the characteristic functions $\phi(\mu)$ of the eigenvalue problems of Case$^{(a)}$ 1 are given by:
\begin{align}\label{Class1}\phi(\mu)&=y'_{\sigma(1)}(a)y_{\sigma(2)}^{(3)}(a)-y_{\sigma(1)}^{(3)}(a)y'_{\sigma(2)}(a)+i\mu^2\left[\beta_3\Big(y_{\sigma(1)}(a)y_{\sigma(2)}^{(3)}(a)\right.\nonumber\\&\quad\left.-y_{\sigma(1)}^{(3)}(a)y_{\sigma(2)}(a)\Big)+\beta_4\Big(y'_{\sigma(1)}(a)y''_{\sigma(2)}(a)-y''_{\sigma(1)}(a)y'_{\sigma(2)}(a)\Big)\right]\nonumber\\&\quad+\beta_3\beta_4\mu^4\Big[y''_{\sigma(1)}(a)y_{\sigma(2)}(a)-y_{\sigma(1)}(a)y''_{\sigma(2)}(a)\Big].\end{align}
Each of the summands in $\phi $ is a product of a power in $\mu$ and a product of two sums of a trigonometric and a hyperbolic functions. The highest $\mu$-power occurs with
\begin{align*}\beta_3\beta_4\mu^4\Big[y''_{\sigma(1)}(a)y_{\sigma(2)}(a)-y_{\sigma(1)}(a)y''_{\sigma(2)}(a)\Big].\end{align*}
Hence we are going to investigate the zeros of 
\begin{align*}\phi_0(\mu)=2\mu^4\Big[y''_{\sigma(1)}(a)y_{\sigma(2)}(a)-y_{\sigma(1)}(a)y''_{\sigma(2)}(a)\Big].\end{align*}
It follows from \eqref{jDerivative} and \eqref{Dcases} that  for the above six cases we obtain: \\
 {\rm Case 1:} $ p_1=0, ~  p_2=1$:
\begin{align*}\phi_0(\mu)=\mu (\cos(\mu a)\sinh(\mu a)-\sin(\mu a)\cosh(\mu a)).\end{align*}
 {\rm Case 2:} $p_1=0,~  p_2=2$:
\begin{align*}\phi_0(\mu)=-\mu^2\sin(\mu a)\sinh(\mu a).\end{align*}
{\rm Case 3:} $ p_1=0,~   p_2=3$: 
\begin{align*} \phi_0(\mu)&=-\mu^3(\sin(\mu a)\cosh(\mu a)+\cos(\mu a)\sinh(\mu a)).\end{align*}
{\rm Case 4:} $p_1=1,~  p_2=2$:
\begin{align*} \phi_0(\mu)&=-\mu^3(\sin(\mu a)\cosh(\mu a)+\cos(\mu a)\sinh(\mu a)).\end{align*}
{\rm Case 5:} $ p_1=1, ~  p_2=3$:
\begin{align*}\phi_0(\mu)&=-2\mu^4\cos(\mu a)\cosh(\mu a).\end{align*}
{\rm Case 6:} $p_1=2, ~ p_2=3$:
\begin{align*} \phi_0(\mu)&=\mu^5 (\sin(\mu a)\cosh(\mu a)-\cos(\mu a)\sinh(\mu a)).\end{align*}

Next  we  give the asymptotic distributions of the zeros of $\phi_0(\mu)$, with proper counting.
 \begin{lem}\label{Class1phizero}
{\rm Case 1:} $ p_1=0, ~  p_2=1$,  $\phi _0$ has a zero of multiplicity  $4$ at $0$, exactly one simple zero $\tilde\mu_k$ in each interval $\left(\left(k-\frac12\right)\frac\pi a,\left(k+\frac12\right)\frac\pi a\right)$ for positive integers $k$ with asymptotics
\begin{equation*}\tilde\mu_k=(4k-3)\frac\pi{4a}+o(1),~k=2,3\dots,\end{equation*}
simple zeros at $\tilde\mu_k$, $-\tilde\mu_k$, $\tilde\mu_{-k}=i\tilde\mu_k$ and $-i\tilde\mu_k$ for $k=2,3, \dots,$ and no other zeros.\\
{\rm Case 2:} $p_1=0,~  p_2=2$, $\phi _0$ has a zero of multiplicity  $4$ at $0$,    simple zeros  at 
\begin{equation*}\tilde{\mu}_{k}= (k-1)\frac{\pi}{a}, ~k=2,3,\dots,
\end{equation*}
simple zeros at $-\tilde\mu_k$, $\tilde\mu_{-k}= i\tilde \mu_k$ and $-i\tilde \mu_k$ for $k=2,3,\dots$, and no other zeros.\\
{\rm Case 3:} $ p_1=0,~   p_2=3$,  $\phi _0$ has a zero of multiplicity  $4$ at $0$, exactly one simple zero 
 $\tilde \mu_k$  in each interval
$\left(\left(k-\frac{1}{2}\right)\frac\pi a,\left(k+\frac{1}{2}\right)\frac \pi a \right)$
for positive integers $k$ with asymptotics
\begin{equation*}\tilde{\mu}_{k}= (4k-5)\frac{\pi}{4a}+o(1), ~k=2,3,\dots,
\end{equation*}
simple zeros at $-\tilde \mu_k$, $\tilde\mu_{-k}= i\tilde \mu_k$ and $-i\tilde \mu_k$ for $k=2,3,\dots$, and no other zeros.\\
{\rm Case 4:} $p_1=1,~  p_2=2$, $\phi _0$ has a zero of multiplicity  $4$ at $0$, exactly one simple zero $\tilde \mu_k$  in each interval
$\left(\left(k-\frac{1}{2}\right)\frac\pi a,\left(k+\frac{1}{2}\right)\frac \pi a \right)$
for positive integers $k$ with asymptotics
\begin{equation*}\tilde{\mu}_{k}= (4k-5)\frac{\pi}{4a}+o(1), ~k=2,3,\dots,
\end{equation*}
simple zeros at $-\tilde\mu_k$, $\tilde\mu_{-k}= i\tilde \mu_k$ and $-i\tilde \mu_k$ for $k=2,3,\dots$, and no other zeros.\\
{\rm Case 5:} $ p_1=1, ~  p_2=3$, $\phi _0$ has a zero of multiplicity $4$ at $0$,  simple zeros at
\begin{equation*}\tilde\mu_k=(2k-1)\frac{\pi}{2a}, ~ k= 2,3,\dots,\end{equation*} 
simple zeros at $-\tilde \mu_k$, $\tilde \mu_{-k}=i\tilde\mu_k$ and $-i\tilde \mu_k$, $k=2,3, \dots$, and no other zeros.\\
{\rm Case 6:} $p_1=2, ~ p_2=3$, $\phi _0$ has a zero of multiplicity $8$ at $0$, exactly one simple zero $\tilde\mu_k$ in each interval $\left(\left(k-\frac12\right)\frac\pi a,\left(k+\frac12\right)\frac\pi a\right)$ for positive integers $k$ with asymptotics
\begin{equation*}\tilde\mu_k=(4k-7)\frac\pi{4a}+o(1),~k=3,4,\dots,\end{equation*}
simple zeros at $\tilde\mu_k$, $-\tilde\mu_k$, $\tilde\mu_{-k}=i\tilde\mu_k$ and $-i\tilde\mu_k$ for $k=3,4,\dots,$ and no other zeros.
 \end{lem}
\begin{proof}The result is obvious in cases 2 and 5. Cases 3 and 4 are identical, while cases  1 and 6 differ in the factor with the power of $\mu$. We will consider Case 3. The choice of the indexing for the non-zeros of $\phi_0$ will become apparent later.

It is easy to see that $\phi_0$ has a zero of multiplicity $4$ at $0$. Next we are going to find the zeros of $\phi_0$ on the positive real axis. One can observe that for $\mu\ne0$, $\phi_0(\mu)=0$ implies $\cosh(\mu a)\ne0$ and $\cos(\mu a)\ne0$, whence the positive zeros of $\phi_0$ are those $\mu>0$ for which $\tan(\mu a)+\tanh(\mu a)=0$. Since $\tan'(x)\ge 1$ and $\tanh'(x)>0$ for all $x\in \mathbb{R}$, the function $\mu \mapsto \tan(\mu a)+\tanh(\mu a)$  is increasing with positive derivative on each interval $\left(\left(k-\frac12\right)\frac{\pi}a, \left(k+\frac12\right)\frac{\pi}a\right)$, $k\in\mathbb{Z}$. On each of these intervals, the function moves from $-\infty$ to $\infty$, thus we have exactly one simple zero $\tilde{\mu}_k$ of  $\tan(\mu a)+\tanh(\mu a)$ in each interval $\left(\left(k-\frac12\right)\frac{\pi}a, \left(k+\frac12\right)\frac{\pi}a\right)$, where $k$ is a positive integer, and no zero in $\left(0,\frac\pi a\right)$. Since $\tanh(\mu a)\to 1$ as $\mu \to\infty$, we have 
\begin{equation*}\tilde\mu_k=(4k-5)\frac\pi{4a}+o(1), ~k=2,3,\dots.\end{equation*} 
The location of the zeros on the other three half-axes follows from repeated application of $\phi_0(i\mu)=-\phi_0(\mu)$.

To complete the proof, we will show that all zeros of $\phi_0$  lie on the real or the imaginary  axis. To this end we observe that the product-to-sum formula for trigonometric functions gives
\begin{align}\label{Case 3}\phi_0(\mu)&=-\mu^3[\sin(\mu a)\cosh(\mu a)+\cos(\mu a)\sinh(\mu a)]\nonumber\allowdisplaybreaks\\&=-\dfrac12\mu^3[\sin((1+i)\mu a)+\sin((1-i)\mu a)-i\sin((1+i)\mu a)\nonumber\allowdisplaybreaks\\&\qquad\qquad+i\sin((1-i)\mu a)]\nonumber\allowdisplaybreaks\\&=-\dfrac12\mu^3[(1-i)\sin((1+i)\mu a)+(1+i)\sin((1-i)\mu a)].\end{align}
Putting  $(1+i)\mu a=x+iy$, $x,~y\in\mathbb{R}$, it follows for $\mu\ne0$ that
\begin{align}\label{prooCase 3}\phi_0(\mu) =0&\Rightarrow |\sin((1+i)\mu a)|=|\sin((1-i)\mu a)|\\\nonumber\allowdisplaybreaks&\Leftrightarrow|\sin(x+iy)|=|\sin(y-ix)|\\\nonumber\allowdisplaybreaks&\Leftrightarrow\cosh^2y-\cos^2x=\cosh^2x-\cos^2y\\\nonumber\allowdisplaybreaks&\Leftrightarrow\cosh^2|y|+\cos^2|y|=\cosh^2|x|+\cos^2|x|.\end{align}
Since $\cosh^2x+\cos^2x=\frac12\cosh (2x)+\frac12\cos(2x)+1$ has a positive derivative on $(0,\infty)$, this function is strictly increasing, and $\phi_0(\mu)=0$ therefore implies by \eqref{prooCase 3} that $|y|=|x|$ and thus $y=\pm x$. Then
\begin{equation*}\mu=\dfrac{x+iy}{(1+i)a}=\dfrac{1\pm i}{1+i}\,\dfrac xa\end{equation*} is either real or pure imaginary.

For Case 1, a power series expansion shows that $\phi_0$ has a zero of multiplicity $4$ at $0$. For the zeros on the positive real axis we just need to replace the function $\mu\mapsto \tan(\mu a)+\tanh(\mu a)$ in the proof of Case 3 by $\mu\mapsto \tan(\mu a)-\tanh(\mu a)$ and observe that $\tanh'(\mu a)<1$. Furthermore, in this case we have a representation of $\phi_0$ similar  to \eqref{Case 3}, except that on the right hand side, the factor $1-i$ in front   of the sine functions are interchanged. Hence \eqref{prooCase 3}  holds in Case 1, and all the zeros must be real or pure imaginary.

Case 6 easily follows from the result for Case 1.\end{proof}
\begin{prop}\label{Class1phizeroProp}For $g=0$, there exists a positive integer $k_0$ such the eigenvalues $\hat\lambda_k$, $k\in\mathbb Z$  of the problems \eqref{eq:2inteq0}, \eqref{case1}, where $(p_3,q_3)=(1,0)$ and $  (p_4,q_4)=(3,2)$,  are  $\hat\lambda_{-k}=-\overline{\hat{\lambda}_k}$,     $\hat\lambda_k=\hat\mu^2_k$ for $k\ge k_0$  and  the $\hat\mu_k$ have the following asymptotic representations as $k\to\infty$:\\\\
{\rm Case 1:} $ p_1=0, ~  p_2=1$,    \quad $ \hat\mu_k=(4k-3)\dfrac\pi{4a}+o(1).$\\\\
{\rm Case 2:} $p_1=0,~  p_2=2$,    \quad $\hat\mu_k=(k-1)\dfrac\pi a+o(1)$.\\\\
{\rm Case 3:} $ p_1=0,~   p_2=3$,  \quad$\hat\mu_k=(4k-5)\dfrac\pi{4a}+o(1)$.\\\\
{\rm Case 4:} $p_1=1,~  p_2=2$,  \quad $\hat\mu_k=(4k-5)\dfrac\pi{4a}+o(1)$.\\\\
{\rm Case 5:} $ p_1=1, ~  p_2=3$,  \quad $\hat\mu_k=(2k-1)\dfrac\pi{2a}+o(1)$.\\\\
{\rm Case 6:} $p_1=2, ~ p_2=3$,  \quad $\hat\mu_k=(4k-7)\dfrac\pi{4a}+o(1)$. \\\\
 In particular, there is an even  number of the pure imaginary eigenvalues  in each case.
\end{prop}
\begin{proof}In each  case, we will show that the zeros of $\phi$ are asymptotically close to the zeros of $\phi_0$. We will start with Case 3.\\
{\rm Case 3:} A straightforward calculation gives
\begin{align}\label{phiCase3}\phi(\mu)&=-\dfrac{\beta_3\beta_4\mu^3}2(\sin(\mu a)\cosh(\mu a)+\cos(\mu a)\sinh(\mu a))\nonumber\allowdisplaybreaks\\
&\quad-\dfrac{i\mu^2\beta_3}2(1-\cos(\mu a)\cosh(\mu a))+\dfrac{i\mu^2\beta_4}2(1+\cos(\mu a)\cosh(\mu a))\nonumber\allowdisplaybreaks\\&\quad+\dfrac{\mu}2(\cos(\mu a)\sinh(\mu a)-\sin(\mu a)\cosh(\mu a)).\end{align}
 Let \begin{align}\label{phi1}\phi_1(\mu)=\dfrac{2\phi(\mu)+\beta_3\beta_4\phi_0(\mu)}{\phi_0(\mu)}.\end{align}
 The first term,  up to the constant $-\frac12\beta_3\beta_4$,   is $\phi_0(\mu)$.  It follows that for $\mu$ with $\phi_0(\mu)\ne0$, $\sin(\mu a)\ne 0$, $\sinh(\mu a)\ne0$, we have 
\begin{align}\label{Case3proofPhi1}\phi_1(\mu)=\dfrac{2\phi(\mu)+\beta_3\beta_4\phi_0(\mu)}{\phi_0(\mu)}&=\dfrac1\mu\,\dfrac{i(\beta_4-\beta_3)}{\cos(\mu a)\cosh(\mu a)}\,\dfrac1{\tan(\mu a)+\tanh(\mu a)}\nonumber\\&-\dfrac1\mu\,\dfrac{i(\beta_4+\beta_3)}{\tan(\mu a)+\tanh(\mu a)}\nonumber\\&+\dfrac1{\mu^2}\,\left[1-\dfrac{2\cos(\mu a)\tanh(\mu a)}{\sin(\mu a)+\cos(\mu a)\tanh(\mu a)}\right].\end{align}

Fix $\varepsilon\in (0,\frac{\pi}{4a})$ and for $k=2,3,\dots$, 
let $R_{k,\varepsilon } $ be the boundaries of the squares determined by the vertices 
$(4k-5)\frac{\pi}{4a}\pm \varepsilon \pm i\varepsilon$. These squares do not intersect due to $\varepsilon<\frac{\pi}{2a}$. 
Since $\tan z=-1$ if and only if $z=j\pi-\frac{\pi}4$ and $j\in\mathbb Z$, it follows from the periodicity of $\tan$ that the number
\begin{align*}C_1(\varepsilon)=2\min\{|\tan(\mu a)+1|:\mu\in R_{k,\varepsilon}\}\end{align*} 
is positive and independent of $\varepsilon$. Since $\tanh(\mu a)\to 1$ uniformly in the strip 
$\{\mu\in\mathbb C: \textrm{ Re }\mu\ge 1, |\textrm{Im }\mu|\le\frac{\pi}{4a}\}$ as $|\mu|\to\infty$, there is and integer $k_1(\varepsilon)$ such that
\begin{align*} |\tan(\mu a)+\tanh(\mu a)|\ge C_1(\varepsilon)~\textrm{ for all } \mu\in R_{k,\varepsilon}  \textrm{ with } k>k_1(\varepsilon).\end{align*}
By periodicity, there is a number $ C_2(\varepsilon)>0$ such that $|\cos(\mu a)|>C_2(\varepsilon)$ for all 
$\mu\in R_{k,\varepsilon}$ and all $k$. Observing $|\cosh(\mu a)|\ge |\sinh(\Re \mu a)|$, it follows that there exists 
$k_2(\varepsilon)\ge k_1(\varepsilon)$ such  that for all $\mu$ on the squares $R_{k,\varepsilon}$ with $k\ge k_2(\varepsilon)$ the estimate $|\phi_1(\mu)|<1 $ holds.
Further we assume from Lemma \ref{Class1phizero} that $\tilde{\mu}_k$ is inside of $R_{k,\varepsilon}$ for $k>k_2(\varepsilon)$ and no other zero of $\phi_0$ has this property. 
By definition of $\phi _1$ in \eqref{phi1} and the estimate $|\phi_1(\mu)|<1$  for all $\mu$ on the square $R_{k, \varepsilon}$,  we have
\begin{align}\label{eq:phiphi1}
|2\phi (\mu )+\beta_3\beta_4\phi_0(mu)|<|\phi_0(\mu)|, 
\end{align}  for all  $ \mu$ on the square $ R_{k,\varepsilon }$.
Hence it follows
by Rouch\'e's theorem that there is exactly one (simple) zero $\hat\mu_k$ of $\phi$ in each $R_{k,\varepsilon }$
for $k\ge k_2(\varepsilon )$. 
In view of $\phi _0(i\mu)=\phi _0(\mu)$ and $\phi _1(i\mu)=-\phi _1(\mu)$ for all $\mu \in \mathbb{C}$, 
the same reasoning applies
 to the corresponding squares along the positive imaginary semiaxis. Observing that $\phi$ is an even function, it follows that 
the same estimate applies to the corresponding squares along the  other remaining two semiaxes. Therefore   
$\phi $ has zeros $\pm\hat\mu_k$, $\pm\hat\mu_{-k}$ for $k
>k_2(\varepsilon )$ with the same asymptotic behaviour as the zeros $\pm \tilde\mu_k$, $\pm i\tilde\mu_k$ of
$\phi_0$ as stated in Lemma \ref{Class1phizero}.

 Next we are going to estimate $\phi_1$ on the squares $S_k$, $k\in \mathbb{N}$, whose vertices are
$\pm k\frac{\pi}{a} \pm ik\frac{\pi}{a}$. For $k\in\mathbb Z$ and $\gamma\in\mathbb R$, 
\begin{equation}\label{tan1}
\tan\left(\left(\frac{k\pi}a+i\gamma\right)a\right)=\tan(i\gamma a)=i\tanh(\gamma a)\in i\mathbb R.
\end{equation}
Therefore, we have for $\mu=\frac{k\pi}a+i\gamma$, where $k\in\mathbb Z$ and $\gamma\in\mathbb R$, that
\begin{equation}\label{tan2}
|\tan(\mu  a)|<1 \textrm{ and } |\tan(\mu a)\pm1|\ge 1.
\end{equation}

For $\mu=x+iy$, $x,y\in\mathbb{R}$ and $x\ne 0$, we have
\begin{equation}\label{tanh1}
\tanh(\mu a)=\frac{e^{(ax+iay)}-e^{-(ax+iay)}}{e^{(ax+iay)}+e^{-(ax+iay)}}\to\pm 1
\end{equation}
uniformly in $y$ as $x\to\pm\infty$.
Hence there is $\widetilde{k}_1>0$ such that  for all $k\in\mathbb{Z}$, $|k|\geq \widetilde{k}_1$,  and $\gamma\in\mathbb{R}$,
\begin{equation}\label{asym16eq9}\Bigl|\tanh\Bigl(\Bigl(\frac{k\pi}{a}+
i\gamma\Bigl)a\Bigl)-\sgn(k)\Bigl|<\frac 12.\end{equation}
It follows from \eqref{tan2} and \eqref{asym16eq9} for $\mu=\frac{k\pi}a+i\gamma$, $k\in\mathbb Z$, $|k|\ge\widetilde{k}_1$, and $\gamma\in\mathbb R$ that
\begin{equation}\label{asym16eq9*}\Bigl|\tan(\mu a)+\tanh(\mu a)\Bigl|\ge \frac 12.\end{equation}
Furthermore, we will make use of the estimates
\begin{align}\label{asym16eqcosh}&\Bigl|\cosh\Bigl(\Bigl(\frac{k\pi}{a}+
i\gamma\Bigl)a\Bigl)\Bigl|\ge|\sinh(k\pi)|,\\
\label{asym16eqcos}&\Bigl|\cos\Bigl(\Bigl(\frac{k\pi}{a}+
i\gamma\Bigl)a\Bigl)\Bigl|=\cosh(\gamma a)\ge1,
\end{align}
which hold for all $k\in\mathbb{Z}$   and all $\gamma\in\mathbb{R}$. Therefore it follows from \eqref{tan2}, 
 \eqref{asym16eq9*}--\eqref{asym16eqcos}     and the corresponding estimates with $\mu$ replaced by $i\mu$ that there is 
$\hat k_1\ge\widetilde k_1$ such that $|\phi_1(\mu)|<1 $ for all $\mu$   $\in S_k$ with $k>\hat k_1$, where $\phi _1$ is as defined in \eqref{phi1}.
By definition of $\phi _1$  in \eqref{phi1}  and the estimate $|\phi_1(\mu)|<1$  for all $\mu$   $\in S_{k}$,
 from  Rouch\'e's theorem we conclude that
the functions $\phi_0$ and $\phi$ have the same number of zeros in the square $S_k$, for $k\in\mathbb{N}$ with $k\ge\hat k_1$.

Since $\phi _0$ has $4k+4$ zeros inside $S_k$ and thus $4k+4+4$ zeros inside $S_{k+1}$, it
follows that $\phi $ has no large zeros other than the zeros $\pm\hat\mu_k$ found above for $|k|
$ sufficiently large, and that there are $\hat \mu_k$ for small $|k|$ such that   
$\hat \lambda _k=\hat\mu_k^2$ account for all eigenvalues of the
problem \eqref{eq:2inteq0}, \eqref{case1}, where $p_1=0$, $p_2=3$, $(p_3, q_3)=(1,0)$ and $(p_4, q_4)=(3,2)$. Since each of these eigenvalues gives rise to two zeros of $\phi $, counted with multiplicity. 
  All eigenvalues with nonzero real part occur in pairs 
$\hat \lambda _k$, $-\overline{\hat\lambda _k}$ with $\Re \hat \lambda _k\ge 0$, which shows that we can index all such eigenvalues as   
$\hat \lambda _{-k}=-\overline{\hat\lambda _k}$. Since there is an even number of remaining indices, the number of pure imaginary eigenvalues must be even.\\
Case 4: The value of  $\sigma(1)$ differs from that in Case 3 by -1 while the value of $\sigma(2)$ differs from that in Case 3 by  1, see \eqref{Dcases}. Hence the function $\phi$ in this case is up to a constant factor  identical to that  in Case 3. Hence the results in cases 3 and 4 are similar. \\
Case  1:  The values of $\sigma(1)$ and $\sigma(2)$ differ from those  in Case 3 by 1.  Hence the function $\phi$ in this case is obtained from that in Case 3 by multiplication by $\mu^{-2}$ and by  replacing each    trigonometric and hyperbolic function by its derivative.  Hence the result follows from that in Case 3.\\
Case  6:  The values of $\sigma(1)$ and $\sigma(2)$ differ from those  in Case 1 by -2.  Hence the function $\phi$ in this case is obtained from that in Case 1 by multiplication by $\mu^4$ and by replacing each   trigonometric   function by its negative. \\
Case 2:  A straightforward calculation gives 
\begin{align}\label{phiCase2}\phi(\mu)&=-\beta_3\beta_4\mu^2\sin(\mu a)\sinh(\mu a)\nonumber\\&\quad+\dfrac{i(\beta_3+\beta_4)\mu}2(\sin(\mu a)\cosh(\mu a)+\cos(\mu a)\sinh(\mu a))\nonumber\\&\quad+\cos(\mu a)\cosh(\mu a).\end{align} Then it follows from \eqref{phi1} that
\begin{align}\label{Case2Phi1}\phi_1(\mu)&=\dfrac{2\phi(\mu)+\beta_3\beta_4\phi_0(\mu)}{\phi_0(\mu)}\nonumber\\& =\dfrac1{2\mu}\,\left(\coth(\mu a)+\cot(\mu a)\right)+\dfrac1{2\mu^2}\,\cot(\mu a)\coth(\mu a).\end{align}
The result follows with reasonings and estimates as in the proof of Case 3, replacing $\mu$ by $\mu\pm\frac\pi2$ and $\mu\pm i\frac\pi2$ respectively.\\
Case 5: Since both $\sigma(1)$ and $\sigma(2)$ differ from the values in Case 2 by -1, it follows that the function $\phi$ in this case is obtained from $\phi$ in Case 2 by multiplication by $\mu^2$ and by replacing the trigonometric and hyperbolic functions by their derivatives. The result follows with reasonings and estimates similar to those in Case 3.
\end{proof}



 \subsection{Asymptotics of eigenvalues for $g=0$ of the  problems of Class Case$^{(a)}$  2}\label{Class3:g=0}

It follows from \eqref{allclasses} and \eqref{phi} that the characteristic functions $\phi(\mu)$ of the eigenvalue problems of Case$^{(a)}$  2  are given by:
\begin{align}\label{Class3}\phi(\mu )&=y''_{\sigma(1)}(a)y_{\sigma(2)}^{(3)}(a)-y_{\sigma(1)}^{(3)}(a)y''_{\sigma(2)}(a)+i\mu^2\left[\beta_3\Big(y'_{\sigma(1)}(a)y_{\sigma(2)}^{(3)}(a)\right.\nonumber\\&\quad\left.-y_{\sigma(1)}^{(3)}(a)y'_{\sigma(2)}(a)\Big)+\beta_4\Big(y''_{\sigma(1)}(a)y_{\sigma(2)}(a)-y_{\sigma(1)}(a)y''_{\sigma(2)}(a)\Big)\right]\nonumber\\&\quad+\beta_3\beta_4\mu^4\left[y_{\sigma(1)}(a)y'_{\sigma(2)}(a)-y'_{\sigma(1)}(a)y_{\sigma(2)}(a)\right].\end{align}
 The highest $\mu$-powers of  the characteristic functions of the problems   of Case$^{(a)}$ 2 occur with
\begin{equation}\label{highestClass3} i\beta_3\mu^2\Big[y'_{\sigma(1)}(a)y_{\sigma(2)}^{(3)}(a)-y_{\sigma(1)}^{(3)}(a)y'_{\sigma(2)}(a)\Big].\end{equation}
Hence we are going to investigate the zeros of 
\begin{equation*}\phi_0(\mu)=2\mu^2\Big[y'_{\sigma(1)}(a)y_{\sigma(2)}^{(3)}(a)-y_{\sigma(1)}^{(3)}(a)y'_{\sigma(2)}(a)\Big].\end{equation*}
It follows from \eqref{jDerivative} and \eqref{Dcases} that  for the   six cases of Case$^{(a)}$  2, we obtain: \\
{\rm Case 1:} $ p_1=0, ~  p_2=1$: 
\begin{align*} \phi_0(\mu)&= \mu (\cos(\mu a)\sinh(\mu a)+\sin(\mu a)\cosh(\mu a)).\end{align*}
{\rm Case 2:} $p_1=0,~  p_2=2$:
\begin{align*} \phi_0(\mu)=\mu^2\cos(\mu a)\cosh(\mu a).\end{align*}
{\rm Case 3:} $ p_1=0,~   p_2=3$:
\begin{align*} \phi_0(\mu)&= \mu^3 (\cos(\mu a)\sinh(\mu a)-\sin(\mu a)\cosh(\mu a)).\end{align*}
{\rm Case 4:} $p_1=1,~  p_2=2$: 
\begin{align*} \phi_0(\mu)&= \mu^3 (\cos(\mu a)\sinh(\mu a)-\sin(\mu a)\cosh(\mu a)).\end{align*}
{\rm Case 5:} $ p_1=1, ~  p_2=3$:
\begin{align*} \phi_0(\mu)&=-2\mu^4\sin(\mu a)\sinh(\mu a).\end{align*}
{\rm Case 6:} $p_1=2, ~ p_2=3$:
\begin{align*} \phi_0(\mu)&=- \mu^5 (\sin(\mu a)\cosh(\mu a)+\cos(\mu a)\sinh(\mu a)).\end{align*}

Next we find the asymptotic distribution of the zeros of the functions $\phi_0$ of the problems of Case$^{(a)}$ 2, with proper counting.
\begin{lem}\label{class3phizero} 
{\rm Case 1:} $ p_1=0, ~  p_2=1$,    $\phi _0$ has a zero of multiplicity  $2$ at $0$, exactly one simple zero $\tilde\mu_k$ in each interval $\left(\left(k-\frac12\right)\frac\pi a,\left(k+\frac12\right)\frac\pi a\right)$ for positive integers $k$ with asymptotics
\begin{equation*}\tilde\mu_k=(4k-1)\frac\pi{4a}+o(1),~k=1,2,\dots,\end{equation*}
simple zeros at $\tilde\mu_k$, $-\tilde\mu_k$, $\tilde\mu_{-k}=i\tilde\mu_k$ and $-i\tilde\mu_k$ for $k=1,2, \dots,$ and no other zeros.\\
{\rm Case 2:} $p_1=0,~  p_2=2$,  $\phi _0$ has a zero of multiplicity $2$ at $0$,  simple zeros at
\begin{equation*}\tilde\mu_k=(2k-1)\frac{\pi}{2a}, ~ k=1,2,\dots,\end{equation*} 
simple zeros at $-\tilde \mu_k$, $\tilde \mu_{-k}=i\tilde\mu_k$ and $-i\tilde \mu_k$, $k=1,2,\dots$, and no other zeros.\\
{\rm Case 3:} $p_1=0$,  $p_2=3$, $\phi _0$ has a zero of multiplicity $6$ at $0$, exactly one simple zero $\tilde\mu_k$ in each interval \allowdisplaybreaks\\ $\left(\left(k-\frac12\right)\frac\pi a, \left(k+\frac12\right)\frac\pi a\right)$ for positive integers $k$ with asymptotics
\begin{equation*}\tilde\mu_k=(4k-5)\frac\pi{4a}+o(1),~k=2,3,\dots,\end{equation*}
simple zeros at $\tilde\mu_k$, $-\tilde\mu_k$, $\tilde\mu_{-k}=i\tilde\mu_k$ and $-i\tilde\mu_k$ for $k=3,4,\dots,$ and no other zeros.\\
  {\rm Case 4:} $p_1=1,~  p_2=2$,   $\phi _0$ has a zero of multiplicity $6$ at $0$, exactly one simple zero $\tilde\mu_k$ in each interval \allowdisplaybreaks\\ $\left(\left(k-\frac12\right)\frac\pi a,\left(k+\frac12\right)\frac\pi a\right)$ for positive integers $k$ with asymptotics
\begin{equation*}\tilde\mu_k=(4k-5)\frac\pi{4a}+o(1),~k=2,3,\dots,\end{equation*}
simple zeros at $\tilde\mu_k$, $-\tilde\mu_k$, $\tilde\mu_{-k}=i\tilde\mu_k$ and $-i\tilde\mu_k$ for $k=3,4,\dots,$ and no other zeros.\\
{\rm Case 5:} $ p_1=1, ~  p_2=3$,  $\phi _0$ has a zero of multiplicity  $6$ at $0$,    simple zeros  at 
\begin{equation*}\tilde{\mu}_{k}= (k-1)\frac{\pi}{a}, ~k=2,3,\dots,
\end{equation*}
simple zeros at $-\tilde\mu_k$, $\tilde\mu_{-k}= i\tilde \mu_k$ and $-i\tilde \mu_k$ for $k=2,3,\dots$, and no other zeros.\\
{\rm Case 6:} $p_1=2, ~ p_2=3$,  $\phi _0$ has a zero of multiplicity  $6$ at $0$, exactly one simple zero $\tilde \mu_k$  in each interval
$\left(\left(k-\frac{1}{2}\right)\frac\pi a,\left(k+\frac{1}{2}\right)\frac \pi a \right)$
for positive integers $k$ with asymptotics
\begin{equation*}\tilde{\mu}_{k}= (4k-5)\frac{\pi}{4a}+o(1), ~k=2,3,\dots,
\end{equation*}
simple zeros at $-\tilde \mu_k$, $\tilde\mu_{-k}= i\tilde \mu_k$ and $-i\tilde \mu_k$ for $k=2,3,\dots$, and no other zeros.\end{lem}
\begin{proof}The proof of Lemma \ref{class3phizero}  is similar to the proof of Lemma \ref{Class1phizero}.\end{proof}

\begin{prop}\label{Class3phizeroProp}For $g=0$, there exists a positive integer $k_0$ such the eigenvalues $\hat\lambda_k$, $k\in\mathbb Z$  of the problems \eqref{eq:2inteq0}, \eqref{case1}, where $(p_3,q_3)=(2,1)$   and  $(p_4,q_4)=(3,0), $ are $\hat\lambda_{-k}=-\overline{\hat{\lambda}_k}$,   $\hat\lambda_k=\hat\mu^2_k$ for $k\ge k_0$  and the $\hat\mu_k$ have the following asymptotic representations as $k\to\infty$:\\\\
{\rm Case 1:} $ p_1=0, ~  p_2=1$,  \quad $ \hat\mu_k=(4k-1)\dfrac\pi{4a}+o(1).$\\\\
{\rm Case 2:} $p_1=0,~  p_2=2$,  \quad $\hat\mu_k=(2k-1)\dfrac\pi {2a}+o(1)$.\\\\
{\rm Case 3:} $ p_1=0,~   p_2=3$, \quad$\hat\mu_k=(4k-5)\dfrac\pi{4a}+o(1)$.\\\\
{\rm Case 4:} $p_1=1,~  p_2=2$,   \quad $\hat\mu_k=(4k-5)\dfrac\pi{4a}+o(1)$.\\\\
{\rm Case 5:} $ p_1=1, ~  p_2=3$, \quad $\hat\mu_k=(k-1)\dfrac\pi{ a}+o(1)$.\\\\
{\rm Case 6:} $p_1=2, ~ p_2=3$,  \quad $\hat\mu_k=(4k-5)\dfrac\pi{4a}+o(1)$.\\\\
 In particular, the number of the pure imaginary eigenvalues is odd in each case.
\end{prop}
\begin{proof}    Case 3:    A straightforward calculation gives 
\begin{align}\label{3phiCase3}\phi(\mu)&=\dfrac{i\beta_3\mu^3}2(\cos(\mu a)\sinh(\mu a)-\sin(\mu a)\cosh(\mu a))\nonumber\allowdisplaybreaks\\&\quad-\dfrac{(1-\beta_3\beta_4)\mu^2}2\sin(\mu a)\sinh(\mu a)\nonumber\allowdisplaybreaks\\&\quad-\dfrac{i\beta_4\mu}2(\sin(\mu a)\cosh(\mu a)+\cos(\mu a)\sinh(\mu a)).\end{align}
All the estimates are as in Case 3 of the proof of Proposition \ref{Class1phizeroProp} and the result  follows from that in Case 3 of the proof Proposition \ref{Class1phizeroProp}.

The results in Case 1, Case 4 and Case 6 follow from reasonings respectively similar to those in Case 1, Case 4 and Case 6  of the proof  of Proposition \ref{Class1phizeroProp}.\\
Case 2:  A straightforward calculation gives 
\begin{align}\label{3phiCase2}\phi(\mu)&=i\beta_3\mu^2\cos(\mu a)\cosh(\mu a)\nonumber\allowdisplaybreaks\\&\quad+\dfrac{(1-\beta_3\beta_4)\mu}2(\cos(\mu a)\sinh(\mu a)-\sin(\mu a)\cosh(\mu a))\nonumber\allowdisplaybreaks\\&\quad-i\beta_4\sin(\mu a)\sinh(\mu a).\end{align}
All the estimates are as in Case 2 of the proof of Proposition \ref{Class1phizeroProp} and the result  follows from that in Case 2 of the proof Proposition \ref{Class1phizeroProp}.

The result  in Case 5 follows from reasonings   similar to those in Case 5    of the proof  of Proposition \ref{Class1phizeroProp}.  \end{proof}


\section{Asymptotics of eigenvalues}\label{gen:g}
Let $D$ be the characteristic function of  the problems  \eqref{eq:2inteq0}, \eqref{case1} for Case$^{(a)}$ 1 and Case$^{(a)}$ 2 with respect to the fundamental system $y_j$,
 $j=1,2,3,4$, with $y_j^{[m]}(0)=\delta_{j,m+1}$ for $m=0,1,2,3$, where $\delta$ is the Kronecker delta.
Denote by $D_0$ the  corresponding characteristic function for $g=0$. Note that the characteristic functions $D_0$ and $\phi_0 $ considered in Section \ref{s:g=0} have the same zeros, counted with multiplicity.
Due to the Birkhoff regularity, $g$ only  influences lower order terms in $D$. Therefore it can be inferred that away from the small squares $R_k$,
 $-R_{k}$,  $i{R}_k$, $-i{R}_{-k}$   around the zeros of $D_0$,
 $|D(\mu)-D_0(\mu)|<|D_0(\mu)|$ if $|\mu|$ is sufficiently large. Since the fundamental system $y_j$,
 $j=1,2,3,4$, depends analytically on $\mu$, also $D$ and $D_0$ are analytic functions.
 Hence applying Rouch\'e's theorem both to the large squares $S_k$   and to the small squares
which are sufficiently far away from the origin, it  follows that the boundary value problem for general $g$ has the same
asymptotic distribution as for $g=0$ for Case$^{(a)}$ 1 and Case$^{(a)}$ 2 respectively. Whence Proposition \ref{Class1phizeroProp}  leads to

\begin{prop}\label{Class1Prop}For $g\in C^{1}[0,a]$, there exists a positive integer $k_0$ such the eigenvalues $\hat\lambda_k$, $k\in\mathbb Z$    of the problem \eqref{eq:2inteq0}, \eqref{case1}, where $B_1(\lambda)y=y^{[p_1]}(0)$, $B_2(\lambda)y=y^{[p_2]}(0)$, $B_3(\lambda)y=y'(a)+i\beta_3\lambda y(a)$, $B_4(\lambda)y=y^{[3]}(a)+i\beta_4\lambda y''(a)$ are  $\hat\lambda_{-k}=-\overline{\hat{\lambda}_k}$,  $\hat\lambda_k=\hat\mu^2_k$ for $k\ge k_0$ and  the $\hat\mu_k$ have the following asymptotic representations as $k\to\infty$:\\\\
{\rm Case 1:} $ p_1=0, ~  p_2=1$,  \quad $ \hat\mu_k=(4k-3)\dfrac\pi{4a}+o(1).$\\\\
{\rm Case 2:} $p_1=0,~  p_2=2$,  \quad $\hat\mu_k=(k-1)\dfrac\pi a+o(1)$.\\\\
{\rm Case 3:} $ p_1=0,~   p_2=3$, \quad$\hat\mu_k=(4k-5)\dfrac\pi{4a}+o(1)$.\\\\
{\rm Case 4:} $p_1=1,~  p_2=2$,   \quad $\hat\mu_k=(4k-5)\dfrac\pi{4a}+o(1)$.\\\\
{\rm Case 5:} $ p_1=1, ~  p_2=3$, \quad $\hat\mu_k=(2k-1)\dfrac\pi{2a}+o(1)$.\\\\
{\rm Case 6:} $p_1=2, ~ p_2=3$,  \quad $\hat\mu_k=(4k-7)\dfrac\pi{4a}+o(1)$. \\\\
 In particular, there is an even  number of the pure imaginary eigenvalues  in each case.
\end{prop}
However  Proposition \ref{Class3phizeroProp} leads to  
\begin{prop}\label{Class3Prop}For $g\in C^{1}[0,a]$, there exists a positive integer $k_0$ such the eigenvalues $\hat\lambda_k$, $k\in\mathbb Z$  of the problem \eqref{eq:2inteq0}, \eqref{case1}, where $B_1(\lambda)y=y^{[p_1]}(0)$, $B_2(\lambda)y=y^{[p_2]}(0)$, $B_3(\lambda)y=y''(a)+i\beta_3\lambda y'(a)$, $B_4(\lambda)y=y^{[3]}(a)+i\beta_4\lambda y(a)$ are   $\hat\lambda_{-k}=-\overline{\hat{\lambda}_k}$, $\hat\lambda_k=\hat\mu^2_k$ for $k\ge k_0$    and the $\hat\mu_k$ have the following asymptotic representations as $k\to\infty$:\\\\
{\rm Case 1:} $ p_1=0, ~  p_2=1$,  \quad $ \hat\mu_k=(4k-1)\dfrac\pi{4a}+o(1).$\\\\
{\rm Case 2:} $p_1=0,~  p_2=2$,  \quad $\hat\mu_k=(2k-1)\dfrac\pi {2a}+o(1)$.\\\\
{\rm Case 3:} $ p_1=0,~   p_2=3$,  \quad$\hat\mu_k=(4k-5)\dfrac\pi{4a}+o(1)$.\\\\
{\rm Case 4:} $p_1=1,~  p_2=2$,   \quad $\hat\mu_k=(4k-5)\dfrac\pi{4a}+o(1)$.\\\\
{\rm Case 5:} $ p_1=1, ~  p_2=3$, \quad $\hat\mu_k=(k-1)\dfrac\pi{ a}+o(1)$.\\\\
{\rm Case 6:} $p_1=2, ~ p_2=3$, \quad $\hat\mu_k=(4k-5)\dfrac\pi{4a}+o(1)$. \\\\
 In particular, there is an odd  number of the pure imaginary eigenvalues  in each case.
\end{prop}
In the remainder of the section we are going to establish more precise eigenvalue asymptotics of the problems of Case$^{(a)}$ 1 and Case$^{(a)}$ 2 respectively.
We again replace $\lambda$ with $\mu^2$. Then according to \cite[Theorem 8.2.1]{men-mol}, \eqref{eq:2inteq0}  has an asymptotic
fundamental system $\{\eta_1,\eta_2,\eta_3,\eta_4\}$ of the form
\begin{equation}\label{eq54can} \eta_{\nu}^{(j)}(x,\mu)=\delta_{\nu,j}(x,\mu)e^{i^{\nu-1}\mu x},\end{equation} where
\begin{equation}\label{eq55can}\delta_{\nu,j}(x,\mu)=\left[ \frac{d^j}{dx^j} \right] \left\{ \sum\limits_{r=0}^4(\mu i^{\nu-1})^{-r}\varphi_r(x)
e^{i^{\nu-1}\mu x} \right\} e^{-i^{\nu -1}\mu x}+o(\mu^{-4+j}),\end{equation}  $j=0,1,2,3$, where
  $[\frac{d^j}{dx^j}]$ means that we omit
those terms of the Leibniz expansion
which contain a function $\varphi_r^{(k)}$ with $k>4-r$.
Since the coefficient of $y^{[3]}$ in
\eqref{eq:2inteq0} is zero, we have $\varphi_0(x)=1$, see \cite[(8.2.3)]{men-mol}.

We will now determine the functions $\varphi_1$ and $\varphi_2$.
In this regard, observe from
\cite[(8.1.2) and (8.1.3)]{men-mol} that $n_0=0$ and $l=4$, see \cite[Theorem 8.1.2]{men-mol}. From \cite[(8.2.45)]{men-mol} we know that
\begin{equation}\label{phir}\varphi _r=\varphi _{1,r}=\varepsilon _1^{\text{\sf T}}V Q^{[r]}
\varepsilon _1,
\end{equation}
 where $\varepsilon_{\nu}$ is the $\nu$-th unit vector in $\C^4$, $V=(i^{(j-1)(k-1)})_{j,k=1}^4$, and  $Q^{[r]}$ are $4\times 4$ matrices given by
\cite[(8.2.28), (8.2.33) and (8.2.34)]{men-mol}, that is, $Q^{[0]}=I_4$,
\begin{align}\label{eq8.2.33*men-mol}
& \Omega_4 Q^{[1]}-Q^{[1]}\Omega_4 =Q^{[0]'}=0,\\
\label{eq8.2.33men-mol}& \Omega_4 Q^{[2]}-Q^{[2]}\Omega_4 =Q^{[1]'}-\frac{1}{4}g\Omega_4\varepsilon \varepsilon^{\top}
\Omega_4^{-2}Q^{[0]},\\
\label{eq8.2.34men-mol}&0=\varepsilon_{\nu}^{\text{\sf T}}\Bigl(Q^{[2]'}+\frac{1}{4}\sum_{j=1}^2k_{3-j}\Omega_4\varepsilon\varepsilon^{\text{\sf T}}
\Omega_4^{-1-j}Q^{[2-j]}\Bigl)\varepsilon_{\nu}\quad (\nu=1,2,3,4),
\end{align}
where  $k_2=-g$, $k_1=-g'$,  $\Omega_4=\diag(1,i,-1,-i)$ and
$\varepsilon^{\text{\sf T}}=(1,1,1,1)$.
Let $G(x)=\int_0^xg(t)dt$. A lengthy but straightforward calculation gives
\begin{equation}\label{varphi1and2}
\varphi _1=\frac14G,\quad \varphi _2=\frac1{32}G^2-\frac18g
\end{equation}
 and thus
\begin{align}\label{eq36can*}
\eta_\nu&=\left(1+\frac{1}{4}i^{-\nu+1}G\mu^{-1}+(-1)^{\nu-1}\left(\frac{1}{32}G^2-\frac{1}{8}g\right)\mu^{-2}\right)e^{i^{\nu-1}\mu x}\nonumber\\&+
\{o(\mu^{-2})\}_{\infty}e^{i^{\nu-1}\mu x}
\end{align}
for $\nu=1,2,3,4$, where $\{o(\ )\}_{\infty}$ means that the estimate is uniform in $x$.

Next we provide  the first four terms of the eigenvalue asymptotics of the problems \eqref{eq:2inteq0}, \eqref{case1} for Case$^{(a)}$ 1 and Case$^{(a)}$ 2 respectively. We are going to start with the problems of Case$^{(a)}$ 1.

The characteristic function of \eqref{eq:2inteq0}, \eqref{case1} for the problems of Case$^{(a)}$ 1 is
$$D(\mu)=\det(\gamma_{j,k}\exp(\varepsilon_{j,k}))_{j,k=1}^4,$$ where
\begin{align*}
&\varepsilon_{1,k}=\varepsilon_{2,k}=0, ~ \varepsilon_{3,k}=\varepsilon_{4,k}=i^{k-1}\mu a,~\gamma_{1,k}=\delta_{k,p_1}(0,\mu),~ 
\\&\gamma_{2,k}=\delta_{k,2}(0,\mu) \text{ if } p_2\le2,
~\gamma_{2,k}=\delta_{k,3}(0,\mu)-g(0)\delta_{k,1}(0,\mu) \text{ if } p_2=3,\\
&\gamma_{3,k}=\delta_{k,1}(a,\mu)+i\beta_3\mu^2\delta_{k,0}(a,\mu),\\
& \gamma_{4,k}=\delta_{k,3}(a,\mu)-g(a)\delta_{k,1}(a,\mu)+i\beta_4\mu^2\delta_{k,2}(a,\mu).\end{align*}
  Note that
\begin{equation}\label{eq2bndg11}D(\mu)=\sum\limits_{m=1}^5\psi_m(\mu)e^{\omega_m \mu a}, \end{equation} where
$\omega_1=1+i$, $\omega_2=-1+i$, $\omega_3=-1-i$, $\omega_4=1- i$, $\omega_5=0$.
The functions $\psi_1,\dots,\psi_5$ have the asymptotics $c_k\mu^k+c_{k-1}\mu^{k-1}+\dots+c_{k_0}\mu^{k_0}+o(\mu^{k_0})$.

It follows from \eqref{eq2bndg11} that
\begin{align}\label{eq8bndg11}D_1(\mu):=D(\mu)e^{-\omega_1\mu a}=\psi_1(\mu)+\sum\limits_{m=2}^5\psi_m(\mu)e^{(\omega_m-\omega_1)\mu a},\end{align}
where $\omega_2-\omega_1=-2$,  $\omega_3-\omega_1=-2-2i$, $\omega_4-\omega_1=-2i$, $\omega_5-\omega_1=-1-i$.

Thus for $\arg \mu\in [-\frac{3\pi}{8},\frac{\pi}{8}]$, we have
$|e^{(\omega_m-\omega_3)\mu a}|\le e^{-\sin \frac{\pi}{8}|\mu|a}$ for $m=2,3,5$ and the terms
$\psi_m(\mu)e^{(\omega_m-\omega_1)\mu a}$ for $m=2,3,5$ can be absorbed by $\psi_1(\mu)$ as
they are of the form $o(\mu^{-s})$ for any integer $s$. Hence for $\arg\mu\in[-\frac{3\pi}{8},\frac{\pi}{8}]$,
\begin{align}
\label{eq8*bndg11}D_1(\mu)=\psi_1(\mu)+\psi_4(\mu)e^{(\omega_4-\omega_3)\mu a}=\psi_1(\mu)+\psi_4(\mu)e^{-2i\mu a},\end{align}
where
\begin{align}\label{eq0bndg31}\psi_1(\mu)
&=\left[ \gamma _{13}\gamma _{24}-\gamma _{23}\gamma _{14}  \right] \,
\left[ \gamma _{31}\gamma _{42}-\gamma _{32}\gamma _{41}  \right],\allowdisplaybreaks\\
\label{eq1bndg31}\psi_4(\mu)
&=\left[ \gamma _{12}\gamma _{23}-\gamma _{22}\gamma _{13} \right] \,
\left[ \gamma _{31}\gamma _{44}-\gamma _{34}\gamma _{41} \right].\allowdisplaybreaks
\end{align}
A straightforward calculation gives
\begin{align}\nonumber\gamma _{31}\gamma _{42}-\gamma _{32}\gamma _{41}
&=2\beta_3\beta_4\mu^6+(1-i)(2\beta_3\beta_4\phi_1(a)+(\beta_3+\beta_4))\mu^5\allowdisplaybreaks\\\label{eq02bndg31}&\quad 
-2i(\beta_3\beta_4\phi^2(a)+(\beta_3+\beta_4)\phi(a)+1)\mu^4+o(\mu^4),\allowdisplaybreaks\\ \nonumber
\gamma _{31}\gamma _{44}-\gamma _{34}\gamma _{41}
&=2\beta_3\beta_4\mu^6+(1+i)(2\beta_3\beta_4\phi_1(a)-(\beta_3+\beta_4))\mu^5\allowdisplaybreaks\\\label{eq12bndg31}&\quad +2i(\beta_3\beta_4\phi^2(a)-(\beta_3+\beta_4)\phi_1(a)+1)\mu^4+o(\mu^4).\allowdisplaybreaks
\end{align}
For the other two factors in \eqref{eq0bndg31} and \eqref{eq1bndg31} we have to consider the six different cases.
\\{\rm Case 1:} $ p_1=0$, $ p_2=1$. We have for this case
\begin{align}\label{Case1psi11} \gamma _{13}\gamma _{24}-\gamma _{23}\gamma _{14}&=
(1-i)\mu +o(\mu)
,\\ \label{Case1psi41} \gamma _{12}\gamma _{23}-\gamma _{22}\gamma _{13}&=
-(1+i)\mu+o(\mu).
\end{align} Therefore
\begin{align}\nonumber \psi_1(\mu)&=2(1-i)\beta_3\beta_4\mu^7- i(\beta_3\beta_4G(a)+2(\beta_3+\beta_4))\mu^6\\
&\quad-\dfrac18(1+i)(G^2(a)-4\beta_3\beta_4g(0)+4(\beta_3+\beta_4)G(a)+16)\mu^5\nonumber\\&+o(\mu^5),\label{psi1case2}
\\\nonumber \psi_4(\mu)&=-2(1+i)\beta_3\beta_4\mu^7-i(\beta_3\beta_4G(a)-2(\beta_3+\beta_4))\mu^6\\&\quad+\dfrac18(1-i)(\beta_3\beta_4G^2(a)-4\beta_3\beta_4g(0)-4(\beta_3+\beta_4)G(a)-16)\mu^5\nonumber\\&\quad+o(\mu^5).\label{psi4case2}\end{align}
{\rm Case 2:} $p_1=0$, $ p_2=2$. Here we get
\begin{align}\label{Case2psi11} \gamma _{13}\gamma _{24}-\gamma _{23}\gamma _{14}&=-2\mu^2+o(\mu^2)
,\\\label{Case2psi41}\gamma _{12}\gamma _{23}-\gamma _{22}\gamma _{13}&=2\mu^2 +o(\mu^2).
\end{align}
Thus 
\begin{align}\nonumber \psi_1(\mu)&=
-4\beta_3\beta_4\mu^8-(1-i)(\beta_3\beta_4G(a)+2(\beta_3+\beta_4))\mu^7
\\&\quad+\dfrac14i(\beta_3\beta_4G^2(a)+(\beta_3+\beta_4)G(a)+4)\mu^6 +o(\mu^6),\label{psi1case1}
\\\nonumber \psi_4(\mu)&=
4\beta_3\beta_4\mu^8+(1+i)(\beta_3\beta_4G(a)-2(\beta_3+\beta_4))\mu^7
\\&\quad+\dfrac14i(\beta_3\beta_4G^2(a)+4(\beta_3-\beta_4)G(a)+16)\mu^6+o(\mu^6).\label{psi4case1}
\end{align}
{\rm Case 3:} $p_1=0$, $ p_2=3$. We obtain
\begin{align}\label{Case3psi11} \gamma _{13}\gamma _{24}-\gamma _{23}\gamma _{14}&=
(1+i)\mu^3+o(\mu^3)
,\\ \label{Case3psi41}\gamma _{12}\gamma _{23}-\gamma _{22}\gamma _{13}&=-(1-i)\mu^3 +o(\mu^3).
\end{align} Hence  
\begin{align}\nonumber \psi_1(\mu)&= 2(1+i)\beta_3\beta_4\mu^9+(\beta_3\beta_4G(a)+2(\beta_3+\beta_4))\mu^8\allowdisplaybreaks\\
&\quad +\dfrac18(1-i)(\beta_3\beta_4G^2(a)-4\beta_3\beta_4g(0)+4(\beta_3+\beta_4)G(a)+16)\mu^7\allowdisplaybreaks\nonumber\\&\quad+o(\mu^7),\label{psi1case3}\\
\nonumber\psi_4(\mu)&=-2(1-i)\beta_3\beta_4\mu^9-(\beta_3\beta_4G(a)-2(\beta_3+\beta_4))\mu^8\allowdisplaybreaks \\
&\quad-\dfrac18(1+i)(\beta_3\beta_4G^2(a)-4\beta_3\beta_4g(0)-4(\beta_3+\beta_4)G(a)+16)\mu^7\allowdisplaybreaks\nonumber \\&\quad+o(\mu^7).\allowdisplaybreaks \label{psi4case3}\end{align}
{\rm Case 4:} $ p_1=1$, $ p_2=2$. Here we have
\begin{align} \label{Case4psi11}\gamma _{13}\gamma _{24}-\gamma _{23}\gamma _{14}&=
(1+i)\mu^3+o(\mu^3),
 \\ \label{Case4psi41}\gamma _{12}\gamma _{23}-\gamma _{22}\gamma _{13}&=
-(1-i)\mu^3 +o(\mu^3).
\end{align} Thus  
\begin{align}\nonumber \psi_1(\mu)&=2(1+i)\beta_3\beta_4\mu^9+(\beta_3\beta_4G(a)+2(\beta_3+\beta_4))\mu^8 \\ &\quad+\dfrac18(1-i)(\beta_3\beta_4G^2(a)-4\beta_3\beta_4g(0)+4(\beta_3+\beta_4)G(a)+16)\mu^7\nonumber\\&\quad+o(\mu^7). \label{psi1case4}\\
\nonumber\psi_4(\mu)&=-2(1-i)\beta_3\beta_4\mu^9-(\beta_3\beta_4G(a)-2(\beta_3+\beta_4))\mu^8 \\&\quad-\dfrac18(1+i)(\beta_3\beta_4G^2(a)-4\beta_3\beta_4g(0)-4(\beta_3+\beta_4)G(a)+16)\mu^7\nonumber \\&\quad+o(\mu^7).\label{psi4case4}\end{align}
{\rm Case 5:} $ p_1=1$, $ p_2=3$. We get
\begin{align} \label{Case5psi11}\gamma _{13}\gamma _{24}-\gamma _{23}\gamma _{14}&=
-2i\mu^4+o(\mu^4),
 \\ \label{Case5psi41}\gamma _{12}\gamma _{23}-\gamma _{22}\gamma _{13}&=
-2i\mu^4+o(\mu^4).
\end{align} Therefore
\begin{align}\nonumber \psi_1(\mu)&= -4i\beta_3\beta_4\mu^{10}-(1+i)(\beta_3\beta_4G(a)+2(\beta_3+\beta_4))\mu^9\\
&\quad -\dfrac14(\beta_3\beta_4G^2(a)+4(\beta_3+\beta_4)G(a)+16)\mu^8+o(\mu^{8}),\label{psi1case5}\\
\nonumber\psi_4(\mu)&=-4i\beta_3\beta_4\mu^{10}+(1-i)(\beta_3\beta_4G(a)-2(\beta_3+\beta_4))\mu^9 \\
&\quad+\dfrac14(\beta_3\beta_4G^2(a)-4(\beta_3+\beta_4)G(a)+16)\mu^8+o(\mu^8). \label{psi4case5}\end{align}
{\rm Case 6:} $ p_1=2$, $ p_2=3$. We obtain
\begin{align} \label{Case6psi11}\gamma _{13}\gamma _{24}-\gamma _{23}\gamma _{14}&=
-(1-i)\mu^5 +o(\mu^5),
\allowdisplaybreaks \\ \gamma _{12}\gamma _{23}-\gamma _{22}\gamma _{13}&=
(1+i)\mu^5 +o(\mu^5). \label{Case6psi41}
\end{align} Hence  
\begin{align}\nonumber \psi_1(\mu)&= -2(1-i)\beta_3\beta_4\mu^{11}+i(\beta_3\beta_4G(a)+2(\beta_3+\beta_4))\mu^{10}\allowdisplaybreaks\\
&\quad +\dfrac18(1+i)(\beta_3\beta_4G^2(a)+12\beta_3\beta_4g(0)+4(\beta_3+\beta_4)G(a)+16)\mu^9\nonumber\\&\quad+o(\mu^9),\label{psi1case6}\allowdisplaybreaks\\
\nonumber\psi_4(\mu)&=2(1+i)\beta_3\beta_4\mu^{11}+i(\beta_3\beta_4G(a)-2(\beta_3+\beta_4))\mu^{10}\allowdisplaybreaks \\
&\quad-\dfrac18(1-i)(\beta_3\beta_4G^2(a)+12\beta_3\beta_4g(0)-4(\beta_3+\beta_4)G(a)+16)\mu^9\allowdisplaybreaks\nonumber\\&\quad+o(\mu^9). \label{psi4case6}\end{align}
We already know  by Proposition \ref{Class1Prop} that the zeros $\mu_k$ of $D$ satisfy the asymptotics $\mu_k=k\frac \pi a+\tau_0+o(1)$ as $k\to \infty $. In order to improve on these asymptotics, write
\begin{equation}\label{eq2bndg31}
\mu_k=k\frac{\pi}{a}+\tau(k),~\tau(k)=\sum\limits_{m=0}^n\tau_{m}k^{-m}+o(k^{-n}), \quad k=1,2,\dots.
\end{equation}
Because of the  symmetry of the eigenvalues, we will only need to find the  asymptotics as $k\to \infty $.
We know $\tau _0$ from Proposition  \ref{Class1Prop},  and our aim is to find $\tau _1$ and $\tau _2$.
To this end we will substitute  \eqref{eq2bndg31} into $D_1(\mu_k)=0$ and we will then compare the coefficients of $k^0$, $k^{-1}$ and $k^{-2}$.

Observe that
\begin{align}\label{eq9bndg31}e^{-2i\mu_k a}&=e^{-2i\tau(k) a}=e^{-2i\tau_{0}a}\exp\left(-2i a\left(\frac{\tau_{1}}{k} +
\frac{\tau_{2}} {k^2}+o(k^{-2})\right)\right)\nonumber\\
&=e^{-2i\tau_{0}a}\left( 1-2ia\tau_{1}\frac1{k}-\left(2a^2\tau _1^2+2ia\tau _2 \right) \frac1{k^2}+o(k^{-2}) \right) , \end{align}
while
\begin{align}\label{eq8bndg31}\frac{1}{\mu_k}=\frac a{\pi k}\left(1+\frac{a\tau(k)}{k\pi}\right)^{-1}=\frac{a}{k\pi}-\frac{a^2\tau _0}{k^2\pi^2}+o(k^{-2}). \end{align}
We know that $D_1(\mu_k)=0$ can be written as
\begin{align}\label{eq9*bndg31}
\mu_k^{-\gamma }\psi_1(\mu_k)+\mu_k^{-\gamma }\psi_4(\mu_k)e^{-2i\tau_k a}=0,\end{align}
where $\gamma $ is the highest $\mu $-power in $\psi _1(\mu)$ and $\psi _4(\mu)$.
Substituting  \eqref{eq9bndg31} and \eqref{eq8bndg31} into   \eqref{eq9*bndg31}  and comparing the coefficients of $k^0$, $k^{-1}$ and $k^{-2}$ we get
\begin{thm}\label{asym41*thm}
For $g\in C^{1}[0,a]$, there exists a positive integer $k_0$ such that  the eigenvalues ${\lambda}_k$,
$k\in\Z$ of the problem \eqref{eq:2inteq0}, \eqref{case1}, where $B_1(y)=y^{[p_1]}(0)$, $B_2(y)=y^{[p_2]}(0)$,
$B_3y=y'(a)+i\beta_3\lambda y(a)$ and $B_4y=y^{[3]}(a)+i\beta_4\lambda y''(a)$ 
are ${\lambda}_{-k}=-\overline{{\lambda}_k}$,  
${\lambda}_k={\mu}_k^2$  for $k\ge k_0$ and the ${\mu}_k$  have the asymptotics
\[\mu_k=k\frac\pi a+\tau _0+\frac{\tau _1}k+\frac{\tau_2}{k^2}+o(k^{-2})\]
and the numbers $\tau _0$, $\tau _1$, $\tau _2$ are as follows:\\
{\rm Case 1:} $ p_1=0$,    $ p_2=1$,
\begin{multline*}\tau_0=-\frac{3\pi}{4a}, \ \tau _1=\frac14\,\frac{G(a)}{\pi }+\frac12\, \frac i{\pi}\,\left(\frac1{\beta_3}+\frac1{\beta_4}\right),\\
\quad\tau _2=\frac 3{16}\,\frac{G(a)}{\pi}-\frac14\,\frac{g(0)}{\pi^2}
 -\frac 14\,\frac a{\pi^2}\left(\frac1{\beta_3^2}+\frac1{\beta_4^2}-\frac2{\beta_3\beta_4}\right)+\frac38\, \frac{i}{\pi}\,\left(\frac1{\beta_3}+\frac1{\beta_4}\right).
\hfill\end{multline*}
{\rm Case 2:} $p_1=0$, $  p_2=2$,
\begin{multline*}\tau_0=-\frac{\pi}{a}, \ \tau _1=  \frac14 \,\frac{G(a)}\pi+\frac12\,\frac i{\pi}\, \left(\frac1{\beta_3}+\frac1{\beta_4}\right), \\
 \quad \tau _2=\frac 14\,\frac{G(a)}\pi-\frac14\,\frac a{\pi^2}\left(\frac 1{\beta_3^2}+\frac 1{\beta_4^2}+\frac 2{\beta_3\beta_4}\right)+\frac12\,\frac i{\pi}\,\left(\frac1{\beta_3}+\frac1{\beta_4}\right).\hfill\end{multline*}
{\rm Case 3:} $p_1=0$, $ p_2=3$,
\begin{multline*}\tau _0= -\frac{5\pi}{4a}, \ \tau _1= \frac14 \,\frac{G(a)}\pi+\frac12\,\frac i{\pi}\, \left(\frac1{\beta_3}+\frac1{\beta_4}\right), \\
\quad\tau _2=\frac 5{16}\,\frac{G(a)}\pi-\frac 14\,\frac{ag(0)}{\pi^2}-\frac14\,\frac a{\pi^2}\left(\frac1{\beta_3^2}+\frac 1{\beta_4^2}-\frac 2{\beta_3\beta_4}\right)+\frac 58\,\frac i\pi\,\left(\frac 1{\beta_3}+\frac 1{\beta_4}\right).
\hfill\end{multline*}
{\rm Case 4:} $ p_1=1$, $p_2=2$,
\begin{multline*}\tau _0= -\frac{5\pi}{4a}, \ \tau _1= \frac14 \,\frac{G(a)}\pi+\frac12\,\frac i{\pi}\, \left(\frac1{\beta_3}+\frac1{\beta_4}\right),\\
\quad\tau _2=\frac 5{16}\,\frac{G(a)}\pi-\frac 14\,\frac{ag(0)}{\pi^2}-\frac14\,\frac a{\pi^2}\left(\frac1{\beta_3^2}+\frac 1{\beta_4^2}-\frac 2{\beta_3\beta_4}\right)+\frac 58\,\frac i\pi\,\left(\frac 1{\beta_3}+\frac 1{\beta_4}\right).\hfill\end{multline*}
{\rm Case 5:} $ p_1=1$, $p_2=3$, 
\begin{multline*}\tau _0= -\frac{\pi}{2a}, \ \tau _1= \frac14 \,\frac{G(a)}\pi+\frac i{2\pi}\, \left(\frac1{\beta_3}+\frac1{\beta_4}\right),\\
\quad\tau _2=\frac 1{8}\,\frac{G(a)}\pi +\frac 14\,\frac i\pi\,\left(\frac 1{\beta_3}+\frac 1{\beta_4}\right)-\frac14\,\frac a{\pi^2}\left(\frac1{\beta_3^2}+\frac 1{\beta_4^2}-\frac 2{\beta_3\beta_4}\right).\hfill\end{multline*}
{\rm Case 6:} $ p_1=2$, $ p_2=3$,
\begin{multline*} \tau _0= -\frac{7\pi}{4a}, \ \tau _1= \frac 14 \,\frac{G(a)}\pi+\frac12\,\frac i{ \pi}\, \left(\frac1{\beta_3}+\frac1{\beta_4}\right),\\
 \quad\tau _2= \frac 7{16}\,\frac{G(a)}{\pi}+\frac {3}{4}\,\frac{ag(0)}\pi+\frac78\,\frac i\pi\,\left(\frac1{\beta_3}+\frac1{\beta_4}\right)-\frac14\,\frac a{\pi^2}\left(\frac1{\beta_3^2}+\frac1{\beta_4^2}-\frac 2{\beta_3\beta_4}\right).\hfill\end{multline*}
 In particular, there is an even  number of the pure imaginary eigenvalues  in each case.
\end{thm}
Next we provide  the first four terms of the eigenvalue asymptotics of the problems \eqref{eq:2inteq0},  \eqref{case1} of   Case$^{(a)}$ 2.

The characteristic function of \eqref{eq:2inteq0}, \eqref{case1} for the problems of Case$^{(a)}$ 2  is
$$D(\mu)=\det(\gamma_{j,k}\exp(\varepsilon_{j,k}))_{j,k=1}^4,$$ where
\begin{align*}
&\varepsilon_{1,k}=\varepsilon_{2,k}=0, ~ \varepsilon_{3,k}=\varepsilon_{4,k}=i^{k-1}\mu a,~\gamma_{1,k}=\delta_{k,p_1}(0,\mu),~ 
\\&\gamma_{2,k}=\delta_{k,2}(0,\mu) \text{ if } p_2\le2,
~\gamma_{2,k}=\delta_{k,3}(0,\mu)-g(0)\delta_{k,1}(0,\mu) \text{ if } p_2=3,\\
&\gamma_{3,k}=\delta_{k,2}(a,\mu)+i\beta_3\mu^2\delta_{k,1}(a,\mu),\\
& \gamma_{4,k}=\delta_{k,3}(a,\mu)-g(a)\delta_{k,1}(a,\mu)+i\beta_4\mu^2\delta_{k,0}(a,\mu).\end{align*}
Note that for the calculations of the functions $\psi_1$ and $\psi_4$ respectively defined in   \eqref{eq0bndg31}  and \eqref{eq1bndg31} only the factors $\gamma _{31}\gamma _{42}-\gamma _{32}\gamma _{41}$ and $\gamma _{31}\gamma _{44}-\gamma _{34}\gamma _{41}$  respectively  given in \eqref{eq02bndg31} and \eqref{eq12bndg31} will change. Hence we are going to provide these two terms. A straightforward calculation gives
\begin{align}\nonumber\gamma _{31}\gamma _{42}-\gamma _{32}\gamma _{41}
&=2\beta_3\mu^6+\frac12(1-i)(\beta_3G(a)-2\beta_3\beta_4+2)\mu^5-\frac18i(\beta_3 G^2(a)\\\label{Cl3eq02bndg31}&\quad 
 +4(1-\beta_3\beta_4)G(a)-16\beta_4)\mu^4+o(\mu^4),\\ \nonumber
\gamma _{31}\gamma _{44}-\gamma _{34}\gamma _{41}
&=-2\beta_3 \mu^6-\frac12(1+i)(\beta_3 G(a)+2\beta_3\beta_4-2)\mu^5-\frac18i(\beta_3 G^2(a)\\\label{Cl3eq12bndg31}&\quad -4\beta_3(1-\beta_4)G(a)-16)\mu^4+o(\mu^4).
\end{align}
Using the same method as for Case$^{(a)}$ 1, we get 
\\{\rm Case 1:} $ p_1=0$,  $ p_2=1$.  It follows from \eqref{eq0bndg31}, \eqref{Case1psi11}, \eqref{Cl3eq02bndg31} on one hand and from  \eqref{eq1bndg31}, \eqref{Case1psi41}, \eqref{Cl3eq12bndg31}  on the other hand that
\begin{align}\nonumber \psi_1(\mu)&=2(1-i)\beta_3 \mu^7- i(\beta_3 G(a)-2\beta_3\beta_4+2)\mu^6-\dfrac18(1+i)(\beta_3G^2(a)\\
&\quad+4(1-\beta_3\beta_4)G(a)-4\beta_3g(0)-16)\mu^5+o(\mu^{5}),\label{Cl3psi1case1}
\\\nonumber \psi_4(\mu)&=2(1+i)\beta_3\mu^7+i(\beta_3 G(a)+2\beta_3 \beta_4-2)\mu^6-\dfrac18(1-i)(\beta_3 G^2(a)\\&\quad-4(1-\beta_3\beta_4)G(a)-4\beta_3g(0)-16\beta_4)\mu^5+o(\mu^5).\label{Cl3psi4case1}
\end{align}
{\rm Case 2:} $p_1=0$, $ p_2=2$. Using \eqref{eq0bndg31}, \eqref{Case2psi11}, \eqref{Cl3eq02bndg31} and  \eqref{eq1bndg31}, \eqref{Case2psi41}, \eqref{Cl3eq12bndg31}, we have 
\begin{align}\nonumber \psi_1(\mu)&=-4\beta_3\mu^8 +(1-i)(\beta_3 G(a)-2\beta_3 \beta_4-2)\mu^7\allowdisplaybreaks\\&\quad-\dfrac14i(\beta_3 G^2(a)+4(1-\beta_3\beta_4)G(a) -16\beta_4)\mu^6 +o(\mu^6),\label{Cl3psi1case2}\allowdisplaybreaks
\\\nonumber \psi_4(\mu)&=-4\beta_3\mu^8 -(1+i)(\beta_3 G(a)+2\beta_3 \beta_4-2)\mu^7\allowdisplaybreaks\\&\quad-\dfrac14i(\beta_3 G^2(a)-4(1-\beta_3\beta_4)G(a) -16\beta_4)\mu^6+o(\mu^6).\label{Cl3psi4case2}\allowdisplaybreaks
\end{align}
{\rm Case 3:} $p_1=0$, $p_2=3$. Putting respectively \eqref{eq0bndg31}, \eqref{Case3psi11}, \eqref{Cl3eq02bndg31} and  \eqref{eq1bndg31}, \eqref{Case3psi41}, \eqref{Cl3eq12bndg31},  together gives   
\begin{align}\nonumber \psi_1(\mu)&=2(1+i)\beta_3\mu^9+ (\beta_3 G(a)-2\beta_3 \beta_4+2)\mu^8+\dfrac18(1-i)(\beta_3 G^2(a)\allowdisplaybreaks\\&\quad+4(1-\beta_3\beta_4)G(a)-4\beta_3g(0) -16\beta_4)\mu^7+o(\mu^7),\label{Cl3psi1case3}\allowdisplaybreaks
\\\nonumber \psi_4(\mu)&=2(1-i)\beta_3\mu^9 +(\beta_3 G(a)+2\beta_3 \beta_4-2)\mu^8+\dfrac18(1+i)(\beta_3 G^2(a)\allowdisplaybreaks\\&\quad-4(1-\beta_3\beta_4)G(a)-\beta_3g(0) -16\beta_4)\mu^7+o(\mu^7).\label{Cl3psi4case3}\allowdisplaybreaks
\end{align}
{\rm Case 4:} $ p_1=1$, $ p_2=2$. The equations  \eqref{eq0bndg31}, \eqref{Case4psi11}, \eqref{Cl3eq02bndg31} and  \eqref{eq1bndg31}, \eqref{Case4psi41}, \eqref{Cl3eq12bndg31},  respectively  yield   
\begin{align}\nonumber \psi_1(\mu)&=2(1+i)\beta_3\mu^9 + (\beta_3 G(a)-2\beta_3 \beta_4+2)\mu^8+\dfrac18(1-i)(\beta_3 G^2(a)\allowdisplaybreaks\\&\quad+4(1-\beta_3\beta_4)G(a)-4\beta_3g(0) -16\beta_4)\mu^7+o(\mu^7),\label{Cl3psi1case4}\allowdisplaybreaks
\\\nonumber \psi_4(\mu)&=2(1-i)\beta_3\mu^9 +(\beta_3 G(a)+2\beta_3 \beta_4-2)\mu^8+\dfrac18(1+i)(\beta_3 G^2(a)\allowdisplaybreaks\\&\quad-4(1-\beta_3\beta_4)G(a)-\beta_3g(0) -16\beta_4)\mu^7+o(\mu^7).\label{Cl3psi4case4}\allowdisplaybreaks
\end{align}
{\rm Case 5:} $ p_1=1$, $ p_2=3$. It follows from  \eqref{eq0bndg31}, \eqref{Case5psi11} and  \eqref{Cl3eq02bndg31}  on one hand and from  \eqref{eq1bndg31}, \eqref{Case5psi41} and from  \eqref{Cl3eq12bndg31} on the other hand that
 \begin{align}\nonumber \psi_1(\mu)&= -4i\beta_3 \mu^{10}+(1+i)(2\beta_3\beta_4G(a)-\beta_3G(a)-2)\mu^{9}\\
&\quad -\dfrac14(\beta_3G^2(a)+4(1-\beta_3\beta_4)G(a)-16\beta_4)\mu^8+o(\mu^8),\label{Cl3psi1case5}\\
\nonumber\psi_4(\mu)&=4i\beta_3\mu^{10}-(1-i)(2\beta_3\beta_4+ \beta_3G(a)-2)\mu^{9} \\
&\quad-\dfrac14(\beta_3 G^2(a)-4(1-\beta_3\beta_4)G(a)-16\beta_4)\mu^8+o(\mu^8). \label{Cl3psi4case5}\end{align}
{\rm Case 6:} $ p_1=2$, $p_2=3$. Using respectively  \eqref{eq0bndg31}, \eqref{Case6psi11} and  \eqref{Cl3eq02bndg31}  on one hand and  \eqref{eq1bndg31}, \eqref{Case6psi41} and  \eqref{Cl3eq12bndg31} on the other hand, we get  
  \begin{align}\nonumber \psi_1(\mu)&= -2(1-i)\beta_3 \mu^{11}+ i( \beta_3 G(a)-2\beta_3\beta_4+2)\mu^{10} +\dfrac18(1+i)(\beta_3G^2(a)\\
&\quad+4(1-\beta_3\beta_4)G(a)+12\beta_3g(0)-16\beta_4)\mu^9+o(\mu^9),\label{Cl3psi1case6}\\
\nonumber\psi_4(\mu)&=-2(1+i)\beta_3 \mu^{11}- i( \beta_3 G(a)+2\beta_3\beta_4-2)\mu^{10} +\dfrac18(1-i)(\beta_3G^2(a)\\
&\quad-4(1-\beta_3\beta_4)G(a)+12\beta_3g(0)-16\beta_4)\mu^9+o(\mu^9). \label{Cl3psi4case6}\end{align}
Using \eqref{eq2bndg31}--\eqref{eq9*bndg31} and applying to Proposition \ref{Class3Prop} the same reasoning and calculations as  for  Proposition \ref{Class1Prop}, we get  
\begin{thm}\label{Cl3asym41*thm}
For $g\in C^{1}[0,a]$, there exists a positive integer $k_0$ such that  the eigenvalues ${\lambda}_k$,
$k\in\Z$   of the problem \eqref{eq:2inteq0}, \eqref{case1}, where $B_1(y)=y^{[p_1]}(0)$, $B_2(y)=y^{[p_2]}(0)$,
$B_3y=y''(a)+i\beta_3\lambda y'(a)$ and $B_4y=y^{[3]}(a)+i\beta_4\lambda y(a)$ 
are    ${\lambda}_{-k}=-\overline{{\lambda}_k}$, 
${\lambda}_k={\mu}_k^2$  for $k\ge k_0$ and the ${\mu}_k$  have the asymptotics
\[\mu_k=k\frac\pi a+\tau _0+\frac{\tau _1}k+\frac{\tau_2}{k^2}+o(k^{-2})\]
and the numbers $\tau _0$, $\tau _1$, $\tau _2$ are as follows:\\
{\rm Case 1:} $ p_1=0$, $ p_2=1$,
\begin{multline*}
\tau_0=-\frac{\pi}{4a}, \ \tau _1=\frac 14\,\frac{G(a)}{\pi }+\frac 1{2}\, \frac i\pi\,\frac{1-\beta_3\beta_4}{\beta_3},\\
\quad \tau _2= \frac1{16}\,\frac{ G (a)}{\pi }-\frac14 \,\frac{ag(0)}{\pi^2} -\frac14\, \frac{a(\beta_3^2\beta_4^2+2\beta_3\beta_4+1)}{\pi^2\beta_3^2}+\frac 1{8 }\, \frac i\pi \frac{1-\beta_3\beta_4}{\beta_3}.
\hfill\end{multline*}
{\rm Case 2:} $p_1=0$, $ p_2=2$,
\begin{multline*}\tau_0=-\frac{\pi}{2a}, \  \tau _1=\frac 14\,\frac{G(a)}{\pi }+\frac 1{2}\, \frac i\pi\frac{1-\beta_3\beta_4}{\beta_3}, \\
 \quad \tau _2=\frac 18\,\frac{G(a)}\pi-\frac14\, \frac{a(\beta_3^2\beta_4^2+2\beta_3\beta_4+1)}{\pi^2\beta_3^2}+\frac 14\, \frac i\pi\, \frac{1-\beta_3\beta_4}{\beta_3}.\hfill\end{multline*}
{\rm Case 3:} $p_1=0$, $ p_2=3$,
\begin{multline*}\tau _0= -\frac{5\pi}{4a}, \ \tau _1=\frac  1{2}\, \frac{\beta_3\beta_4-1}{\pi\beta_3}-\frac i4\,\frac{G(a)}{\pi },\\
\quad\tau _2=-\frac 1{16}\,\frac{aG^2(a)}{\pi^2}+\frac 14\,\frac{aG(a)}{\pi^2}\,\frac{\beta_4\beta_3- 1}{\beta_3} +\frac{5}8\,  \frac{\beta_4\beta_3-1}{\pi\beta_3}-\frac14\,\frac{a(\beta_3^2\beta_4^2-2\beta_3\beta_4+1)}{\pi^2\beta_3^2}\\\qquad\quad-\frac{5i}{16}\, \frac{G(a)}\pi.
\hfill\end{multline*}
{\rm Case 4:} $p_1=1$, $ p_2=2$,
\begin{multline*}\tau _0= -\frac{5\pi}{4a}, \ \tau _1=\frac  1{2}\, \frac{\beta_3\beta_4-1}{\pi\beta_3}-\frac i4\,\frac{G(a)}{\pi },\\
\quad\tau _2=-\frac 1{16}\,\frac{aG^2(a)}{\pi^2}+\frac 14\,\frac{aG(a)}{\pi^2}\,\frac{\beta_4\beta_3- 1}{\beta_3} +\frac{5}8\,  \frac{\beta_4\beta_3-1}{\pi\beta_3}-\frac14\,\frac{a(\beta_3^2\beta_4^2-2\beta_3\beta_4+1)}{\pi^2\beta_3^2}\\\qquad\quad-\frac{5i}{16}\, \frac{G(a)}\pi.
\hfill\end{multline*}
{\rm Case 5:} $ p_1=1$, $p_2=3$, 
\begin{multline*}\tau _0= -\frac{\pi}{a}, \ \tau _1=\frac 14\,\frac{G(a)}{\pi }+\frac 1{2}\, \frac i\pi\, \frac{1-\beta_3\beta_4}{\beta_3},\\
\quad \tau _2=\frac 14\,\frac{G(a)}\pi -\frac14\,\frac{a(\beta_3^2\beta_4^2+2\beta_3\beta_4+1)}{\pi^2\beta_3^2}- \frac{ 1}2\, \frac i\pi\, \frac{1-\beta_3\beta_4}{ \beta_3}.\hfill\end{multline*}
{\rm Case 6:} $ p_1=2$, $ p_2=3$,
\begin{multline*}\tau _0= -\frac{5\pi}{4a}, \ \tau _1=\frac 14\,\frac{G(a)}{\pi }+\frac 1{2}\, \frac i\pi\,  \frac{1-\beta_3\beta_4}{\pi\beta_3},\\
\quad \tau _2=\frac 5{16}\,\frac{G(a)}\pi+\frac 34\,\frac{ag(0)}{\pi^2} -\frac14\,\frac{a(\beta_3^2\beta_4^2+2\beta_3\beta_4+1)}{\pi^2\beta_3^2}+\frac{5 }8\, \frac i\pi\, \frac{1-\beta_3\beta_4}{\pi\beta_3}.
\hfill\end{multline*}
 In particular, there is an odd number of pure imaginary eigenvalues in each case.
\end{thm}
Note that the functions $\psi_1$ and $\psi_4$ in Case  3  of  the problems of classes Case$^{(a)}$ 1 and Case$^{(a)}$ 2 are respectively equal to those of Case 4 of the same class. Hence the values of $\tau_k$, $k=0,1,2$ in Case 3 of each of the classes Case$^{(a)}$ 1 and Case$^{(a)}$ 2 are equal to those in Case 4 for the corresponding class.








\begin{rmk}\label{Finalrmk}\rm
In \cite{MolZin1} we have considered the differential equation \eqref{eq:2inteq0}  with the boundary terms $B_1(\lambda)y$ and $B_2(\lambda)y$ at $0$ as in this paper,   but only  the cases 1,2,5  and 6.

The boundary  terms $B_3(\lambda)y$ and $B_4(\lambda)y$ considered in Case$^{(a)}$ 1 of this paper differ from those of \cite{MolZin1}. However according to the values of $\tau_1$,  we can observe that  if $\beta_j>0$, $j=3,4$,  or if   $\beta_3\beta_4<0$    and $\beta_3+\beta_4\le 0$, then  the eigenvalues of the operator pencil $L(\lambda)$ lie on the closed upper half-plane satisfying \cite[Proposition 2.3]{MolZin1}.

The boundary terms  $B_3(\lambda)y$ and $B_4(\lambda)y$  considered  in Case$^{(a)}$ 2 of this paper are those of \cite{MolZin1} but  where $\beta_3>0$ and $\beta_4<0$.  
We can observe that all eigenvalues of $L(\lambda)$ lie in the closed upper half-plane   in cases 1, 2, 5 and 6 if $\beta_3>0$ and $\beta_3\beta_4<1$ or  if   $\beta_3<0$ and $\beta_3\beta_4>1$ . 
However  the eigenvalues in cases 3 and 4 will  lie  in the closed upper half-plane if $\beta_3>0$ and $\beta_3\beta_4>1$    or    $\beta_3<0$ and $\beta_3\beta_4<1$.

\end{rmk}

 \section{Asymptotics of eigenvalues of the problem describing the stability of a flexible missile}\label{flex}
In this section we consider the problem \eqref{eq:2inteq0}, \eqref{case1} where  $\beta_3=\beta_4=0$, $p_1=p_3=2$ and $p_2=p_4=3$.  It follows from \eqref{jDerivative},  \eqref{Dcases} and \eqref{Class3} that the characteristic function of the problem for $g=0$ is: 
\begin{equation}\label{flexChar}\phi(\mu)=2\mu^4[1-\cos(\mu a)\cosh(\mu a)].\end{equation}
Next we give the asymptotics distributions of the zeros of $\phi(\mu)$ with their proper counting.
\begin{lem}\label{fexphiLem}
For $g=0$ the function $\phi$ has a zero of multiplicity eight at $0$, exactly one simple zero in each interval  
$[ 2m\frac{\pi}a,(2m+\frac12)\frac{\pi}a]$ and  $[(2m+\frac32)\frac{\pi}a, (2m+2)\frac{\pi}a]$, respectively,  for  nonnegative integers $m$ with asymptotics
\begin{align*}\tilde{\mu}_k=(2k-5)\frac{\pi}{2a}+o(1),\quad k=3,4,\dots, \end{align*}simple zeros at $-\tilde{\mu}_k$, $\tilde{\mu}_{-k}=i\tilde{\mu}_k$, -$i\tilde{\mu}_k$, for $k=3,4,\dots$, and no other zeros. \end{lem}
\begin{proof}
It is easy to see that 0 is a zero of $\phi$ of multiplicity 8.  Next we  find the zeros of $\phi$ on the positive real axis. Let $f(\mu)= \cos(\mu a)\cosh(\mu a)-1$ and $I_{m,j}=[(2m+\frac j2)\frac \pi a, (2m+\frac{j+1}2)\frac\pi a]$, $m=0,1,\dots$, $j=0,1,2,3$. The zeros of $\phi$ are the zeros of $f$. It  is obvious that for all $m$ and $\mu\in I_{m,1}\cup I_{m,2}$, $f(\mu)\le 1$. On $I_{m,1}$, $\mu\mapsto \cos(\mu a)$ is decreasing and positive, while $\mu\mapsto \cosh(\mu a)$ is increasing and positive, so that $f$ is increasing. At the endpoints of this interval, $f$ has the values $f(2m\frac\pi a)=\cosh (2m\pi)-1>0$ and $f((2m+\frac12)\frac\pi a)=-1<0$. Hence $f$ has exactly one simple zero on $I_{m,0}$.  From $f''(\mu)=-2a^2\sin(\mu a)\sinh(\mu a)$ we see that $f$ is strictly convex on $I_{m,3}$
with $f((2m+\frac32)\frac{\pi}a)=-1<0$ and $f((2m +2)\frac{\pi}a)=\cosh( (2m+2)\pi)-1>0$. Hence $f$ has exactly one simple zero on $I_{m,3}$. Since $\dfrac 1{\cosh(\mu a)}\to 0$ as $\mu \to\infty$, we
have
\begin{align*}\tilde{\mu}^1_m=\left(2m+\frac12\right)\frac{\pi}{a}+o(1) \text{ and }
\tilde{\mu}^2_m=\left(2m+\frac32\right)\frac{\pi}{a}+o(1),\quad m=0,1,\dots.\end{align*}
The location of the zeros on the other three half-axes follows by repeated  application of $\phi (i\mu)=\phi (\mu)$.

To complete the proof we show that all zeros of $\phi$ lies on  the real or the imaginary axis. Define the eigenvalue problem
\begin{align}\label{flexLemEq0}y^{(4)}=\tau y, ~y''(0)=0,~y^{(3)}(0)=0,~y''(a)=0, ~y^{(3)}(a)=0.\end{align} The substitution of $\tau=\mu^4$ shows that $\mu\mapsto -2\mu^4f(\mu)$  is the characteristic function of the problem \eqref{flexLemEq0}. Hence the zeros of $f$ are fourth roots of   nonnegative real numbers, which means that all zeros of $f$ are real or pure imaginary. \end{proof}
\begin{prop}\label{flexprop}
For $g=0$, $\beta_3=\beta_4=0$,   $p_1=p_3=2$ and $p_2=p_4=3$, there exists a positive integer $k_0$ such that  the eigenvalues $\hat{\lambda}_k$, $k\in\mathbb{Z}$,
counted with multiplicity, of the problem \eqref{eq:2inteq0}--\eqref{case1},   can be indexed in such a way that the eigenvalues $\hat{\lambda}_k$ are real and satisfy $\hat{\lambda}_{-k}=- \hat{\lambda}_k $. For $k>0$, we can write
$\hat{\lambda}_k=\hat{\mu}_k^2$, where  the $\hat{\mu}_k$  have the following asymptotic representation  as $k\to \infty $:\\
$$\hat\mu_{k}= (2k-5)\frac{\pi}{2a}+o(1).$$
\end{prop}
Note that in this case, there is no perturbed term. Hence $\phi_1(\mu)=0$  and $\phi(\mu)=\phi_0(\mu)$.

The characteristic function of \eqref{eq:2inteq0}, \eqref{case1}, in this case,  is
$$D(\mu)=\det(\gamma_{j,k}\exp(\varepsilon_{j,k}))_{j,k=1}^4,$$ where
\begin{align*}
&\varepsilon_{1,k}=\varepsilon_{2,k}=0, ~ \varepsilon_{3,k}=\varepsilon_{4,k}=i^{k-1}\mu a,~
\\& \gamma_{1,k}=\delta_{k,2}(0,\mu),~ \gamma_{2,k}=\delta_{k,3}(0,\mu)-g(0)\delta_{k,1}(0,\mu),\\
&\gamma_{3,k}=\delta_{k,2}(a,\mu), \,
 \gamma_{4,k}=\delta_{k,3}(a,\mu)-g(a)\delta_{k,1}(a,\mu).\end{align*}
We are going to calculate of the functions $\psi_1$ and $\psi_4$ respectively defined in   \eqref{eq0bndg31}  and \eqref{eq1bndg31}.   A straightforward calculation gives
\begin{align}\label{flexpsi11} \gamma _{13}\gamma _{24}-\gamma _{23}\gamma _{14}&=-(1-i)\mu^5+\frac34(1+i)g(0)\mu^3+o(\mu^3)
,\\\label{flexpsi41}\gamma _{12}\gamma _{23}-\gamma _{22}\gamma _{13}&= (1+i)\mu^5-\frac34(1-i)g(0)\mu^3+o(\mu^3),
\end{align}
\begin{align}\nonumber\gamma _{31}\gamma _{42}-\gamma _{32}\gamma _{41}
&=(1-i)\mu^5+\frac12iG(a)\mu^4\allowdisplaybreaks\\\label{flexeq02bndg31}&\quad 
-\frac1{16}(1+i)\left(G^2(a)+12g(a)\right)\mu^3 +o(\mu^3),\allowdisplaybreaks\\ \nonumber
\gamma _{31}\gamma _{44}-\gamma _{34}\gamma _{41}
&=(1+i)\mu^5+\frac12iG(a)\mu^4\allowdisplaybreaks\\\label{flexeq12bndg31}&\quad -\frac1{16}(1-i)(G^2(a)+12g(a))+o(\mu^3).\allowdisplaybreaks
\end{align}
Therefore it follows from \eqref{eq0bndg31} and \eqref{eq1bndg31} that
 \begin{align}\psi_1(\mu)&= 2i\mu^{10}+\frac12(1+i)G(a)\mu^{9}\nonumber\\&\qquad+\frac18(G^2(a)+12(g(0)+g(a))\mu^{8}+o(\mu^{8}),\label{Flexpsi1}\\
  \psi_4(\mu)&=2i\mu^{10}-\frac12(1-i)G(a)\mu^{9}\nonumber\\&\qquad+\frac18(G^2(a)+12(g(0)+g(a))\mu^{8}+o(\mu^{8}). \label{flexpsi4}\end{align}
Using \eqref{eq2bndg31}--\eqref{eq9*bndg31} and applying to Proposition \ref{flexprop} the same reasoning and calculations as for  Proposition \ref{Class1Prop}, we get  
\begin{thm}\label{flexThm} For $g\in C^{1}[0,a]$, there exists a positive integer $k_0$ such that  the eigenvalues ${\lambda}_k$,
$k\in\Z$   of the problem  describing the stability of a flexible missile   are    ${\lambda}_{-k}=- {\lambda}_k $, 
${\lambda}_k={\mu}_k^2$  for $k\ge k_0$ and the ${\mu}_k$  have the asymptotics
\[\mu_k=k\frac\pi a+\tau _0+\frac{\tau _1}k+\frac{\tau_2}{k^2}+o(k^{-2})\]
and the numbers $\tau _0$, $\tau _1$, $\tau _2$ are as follows: 
\begin{multline*}\tau _0= -\frac{5\pi}{2a}, \ \tau _1=\frac 14\,\frac{G(a)}{\pi },
\ \tau _2=\frac 5{8}\,\frac{G(a)}{\pi^2}+\frac 14\, \frac a{\pi^2}(5g(0)+3g(a)).\hfill\end{multline*}
 In particular,all the eigenvalues are real.
 \end{thm}
Note from Lemma \ref{fexphiLem} and the values of $\tau_0$, $\tau_1$ and $\tau_2$ in  Theorem \ref{flexThm} that the asymptotics of the zeros of $\phi(\mu)$ defined in \eqref{flexChar} are either real or pure imaginary. Hence the eigenvalues of the problem describing the stability of a flexible missile are all real.

Note as well that according to \cite[Theorem 1.2]{MolZin} the problem describing the stability of a flexible missile is self-adjoint and therefore its eigenvalues must necessary be real.

\textbf{Acknowledgement.} 
I would like to thank Prof Manfred M\"oller for fruitful discussions.   
%

\section*{References}

\bibliography{mybibfile}

\end{document}